\documentclass[11pt]{amsart}
\usepackage{amssymb,amsmath,amsthm}
\usepackage{amsfonts}
\usepackage{enumerate}
\usepackage{dsfont}
\usepackage{cite}
\usepackage[colorlinks, citecolor=red]{hyperref}
\usepackage{mathrsfs}
\usepackage{epsfig}
\usepackage{lscape}
\usepackage{subfigure}
\usepackage{epstopdf}
\usepackage{caption}
\usepackage{algorithm}
\usepackage{algpseudocode}
\usepackage{multirow}
\usepackage{geometry}
\usepackage{paralist}
\usepackage{enumerate}
\usepackage{url}
\usepackage{graphicx,cite}
\usepackage{longtable}
\usepackage{tikz}
\usepackage{adjustbox}
\usepackage{CJKutf8}

\newtheorem{definition}{Definition}[section]
\newtheorem{corollary}[definition]{Corollary}

\newtheorem{theorem}[definition]{Theorem}
\newtheorem{lemma}[definition]{Lemma}

\newtheorem{remark}[definition]{Remark}

\newtheorem{assumption}{Assumption}[section]
\date{}

\begin{document}
	%\begin{sloppypar} %自动换行、对齐
	\begin{CJK}{UTF8}{gbsn}
		\baselineskip 18pt
		\bibliographystyle{plain} %参考文献；其中plain是按字母的次序排列
		
		\title[ASEIM for solving least squares]{On adaptive stochastic extended iterative methods for solving least squares}
		
		\author{Yun Zeng}
		\address{School of Mathematical Sciences, Beihang University, Beijing, 100191, China. }
		\email{zengyun@buaa.edu.cn}
		
		\author{Deren Han}
		\address{LMIB of the Ministry of Education, School of Mathematical Sciences, Beihang University, Beijing, 100191, China. }
		\email{handr@buaa.edu.cn}
		
		\author{Yansheng Su}
		\address{School of Mathematical Sciences, Beihang University, Beijing, 100191, China. }
		\email{suyansheng@buaa.edu.cn}
		
		\author{Jiaxin Xie}
		\address{LMIB of the Ministry of Education, School of Mathematical Sciences, Beihang University, Beijing, 100191, China. }
		\email{xiejx@buaa.edu.cn}

		\begin{abstract}
			In this paper, we propose a novel adaptive stochastic extended iterative method,  which can be viewed as an improved extension of the randomized extended Kaczmarz (REK) method, for finding the unique minimum Euclidean norm least-squares solution of a given linear system. In particular, we introduce three equivalent stochastic reformulations of the  linear least-squares problem: stochastic unconstrained and constrained optimization problems, and the stochastic multiobjective optimization problem. We then alternately employ  the adaptive variants of the stochastic heavy ball momentum (SHBM) method, which utilize iterative information to update the parameters, to solve the stochastic reformulations.
			We prove that our method converges $R$-linearly in expectation, addressing an open problem in the literature related to designing theoretically supported adaptive SHBM methods. Numerical experiments show that our adaptive stochastic extended iterative method has strong advantages over the non-adaptive one.
		\end{abstract}
		
		\maketitle
		
		\let\thefootnote\relax
		
		\let\thefootnote\relax
		
		%----------------------------section introduction--------------------------

		\section{Introduction}
		\label{section_1}
		Large-scale linear systems arise in many areas of scientific computing  and engineering, including computerized tomography \cite{natterer2001mathematics}, machine learning \cite{Cha08}, signal processing \cite{Byr04}, etc. In this paper, we consider solving the linear system 
		\begin{equation}
			\label{LS}
			Ax=b, \ A\in\mathbb{R}^{m\times n}, \ b\in\mathbb{R}^m
		\end{equation}
		in the least-squares sense, i.e., finding $x^*$ such that
		\begin{equation}
			\label{LS_norm}
			x^*\in \arg\min \{\|Ax-b\|^2_2, x\in\mathbb{R}^n\}.
		\end{equation}
		It is well-known \cite{golub2013matrix,ben2003generalized,Du20Ran} that the problem \eqref{LS_norm} has a unique minimum Euclidean norm least-squares solution  $A^\dagger b$, where $A^\dagger$ denotes the Moore-Penrose pseudoinverse of $A$.  
		This solution exists regardless of the dimensions $m$ and $n$, the rank of $A$, or the consistency of the linear system.
		% In this work, we aim to efficiently compute this solution across all possible scenarios.}  
		
		The randomized Kaczmarz (RK) method \cite{Str09} has recently gained popularity in solving large-scale linear systems for its reduced iteration cost and low memory storage requirements. 
		If the linear system \eqref{LS} is consistent, the RK method converges linearly in expectation to $A^\dagger b+(I-A^\dagger A)x^0$ with an arbitrary initial point $x^0$ \cite{Str09,Han2022-xh, DS2021}. In particular, if $x^0\in\operatorname{Range}(A^\top)$, then the RK method can obtain the unique minimum Euclidean norm least-squares solution  $A^\dagger b$. Here, \(\top\) denotes the transpose of a matrix or a vector, and \(\operatorname{Range}(A^\top)\) denotes the range space of \(A\).
		If it is inconsistent, however, Needell \cite{needell2010randomized} showed that the RK method only converges within a radius (convergence horizon) of the least-squares solution; see also \cite{bai2021greedy,ma2015convergence} for some further comments.
		
		To address this issue, Zouzias and Freris \cite{Zou12} applied a modification to the RK method, proposing the \emph{randomized  extended  Kaczmarz} (REK) method. Letting $\text{Range}(A)$ represent the range space of $A$ and $b_{\text{Range}(A)}$ denote the orthogonal projection of $b$ on $\operatorname{Range}(A)$, 
		it is known that \eqref{LS_norm} is equivalent to solving the consistent system $Ax=b_{\text{Range}(A)}$ \cite{hanke1990acceleration}.
		Therefore, the REK method addresses the inconsistency by constructing an auxiliary sequence $z^k$ such that $b-z^k$ approximates \( b_{\operatorname{Range}(A)}\).
		Starting from \(z^0 \in b + \text{Range}(A)\) and \(x^0 \in \text{Range}(A^{\top})\), the REK method updates $z^k$ and $x^k$  alternately, by solving \( A^\top z = 0 \) and \( Ax = b - z^k \) with typical RK method, respectively. 
		The convergence property of the RK method ensures that $b-z^k$ converges to  \( b_{\operatorname{Range}(A)}\) (see Lemma \ref{z-the basic method-convergence}), and consequently $x^k$ converges to $A^\dagger b$. Remark \ref{rek-iteration} provides the concrete iteration scheme. 
		In this paper, we employ the same procedure to propose a more general framework.
		Furthermore, we  also intend to incorporate the heavy ball momentum (HBM) \cite{polyak1964some,ghadimi2015global} into the  REK-type method to enhance its performance.  To the best of our knowledge, this is the first study to investigate the momentum variants of the REK-type method.

		\subsection{Our contributions}
		
		In this paper, we present a simple and versatile stochastic extended iterative framework for computing the unique minimum Euclidean norm least-squares solution to linear systems.  The main contributions of this work are as follows.
		
		\begin{itemize}
		\item[1.]
		We introduce a novel scheme to reformulate any least-squares problem  into several seemingly different but equivalent forms, including the augmented linear system \cite{arioli1989augmented, bai2021greedy}, the unconstrained and constrained stochastic optimization problem, and the multiobjective optimization problem. This reformulation enables us to establish new connections between these problems. Moreover, we identify the conditions that ensure these reformulations share the same solution set.       
		\item[2.]
		It is well-known that one of the limitations of the HBM method is its dependence on  prior knowledge of certain problem parameters, such as singular values of a certain matrix \cite{loizou2020momentum,Han2022-xh,ghadimi2015global,bollapragada2022fast}. 
		If  $Ax=b$ is consistent, there have been works on the strategy to adaptively learn the parameters in the stochastic HBM (SHBM) method, such as the adaptive SHBM (ASHBM) method \cite{zeng2023adaptive} and the adaptive Bregman-Kaczmarz method \cite{lorenz2023minimal,zeng2023fast}. However, these strategies cannot be directly applied to general cases of the least-square problems. In this paper, we address the issues brought by inconsistency and extend such adaptive strategies, proposing the adaptive stochastic extended iterative method. To the best of our knowledge, this is the first time that adaptive HBM is integrated for solving the least-squares problem \eqref{LS_norm} without assuming the consistency of the corresponding linear system $Ax=b$. 
		\item[3.] 
		We demonstrate that our method converges $R$-linearly in expectation to the unique minimum Euclidean norm least-squares solution $A^{\dagger}b$. 
		Furthermore, our results hold for a wide range of  probability spaces, this flexibility enables the development of new, efficient variants tailored to problems with specific structures through the design of suitable probability spaces. Numerical experiments are also provided to confirm our results.
		\end{itemize}
		
		\subsection{Related work}
		
		\subsubsection{Stochastic reformulation of linear systems} 
		
		In the seminal paper \cite{richtarik2020stochastic}, Richt{\'a}rik and Tak{\'a}{\v{s}}  introduced the concept of stochastic reformulation of consistent linear systems.
		This approach, offering multiple equivalent interpretations, can serve as a platform for researchers from various communities to leverage their domain-specific insights.
		Specifically, their reformulation includes as a stochastic optimization problem, a stochastic linear system, a stochastic fixed point problem, and a stochastic intersection problem. They explored methods including stochastic gradient descent, stochastic Newton's method, stochastic proximal point method, stochastic fixed point method, and stochastic projection method, demonstrating that these methods are equivalent when applied to the corresponding reformulations. 
		However, the inclusion of the Moore-Penrose inverse in their reformulations would make the associated algorithm difficult to parallelize.
		
		Recently, Zeng et al. \cite{zeng2023adaptive} proposed a novel scheme to reformulate consistent linear system into a stochastic problem and employed the SHBM to solve the associated problem, where parallel computing is able to be applied. 
		Subsequently, Lorenz and Winkler \cite{lorenz2023minimal} adopted this idea and developed the minimal error momentum Bregman-Kaczmarz method to find a solution to the linear system with certain properties, such as sparsity.
		However, all of stochastic reformulations mentioned above are based on the assumption that the linear system is consistent. This paper goes beyond this limitation and develop a family of stochastic reformulations for arbitrary linear systems (consistent or inconsistent, overdetermined or underdetermined, full-rank or rank-deficient). To the best of our knowledge, this is the first study that addresses the reformulation of any linear system into a stochastic problem.
		
		\subsubsection{The Kacmzarz method}
		\label{section-122}
		
		The Kaczmarz method \cite{Kac37}, also known as algebraic reconstruction technique (ART) \cite{herman1993algebraic,gordon1970algebraic}, is a classic yet effective  row-action iteration solver for the linear system. Starting from  $x^0\in\mathbb{R}^n$, the Kaczmarz method constructs $x^{k+1}$ by
		%%%%%%ADV: i or ik
		$$
		x^{k+1}=x^k-\frac{ A_{i_k, :},x^k-b_{i_k}}{\|A_{i_k, :}\|^2_2}A_{i_k, :}^\top,
		$$
		where $ A_{i_k, :} $ denotes the $ i_k $-th row of $ A $, $ b_{i_k} $ denotes the $ i_k $-th entry of $ b $, and $ {i_k} $ is cyclically selected from $ [m]:=\{1,\cdots,m \} $. The iteration sequence $\{x^k\}_{k\geq 0}$ converges to a certain solution but the convergence rate is hard to obtain.  In the seminal paper \cite{Str09}, Strohmer and Vershynin studied the RK method and proved that when the linear system is consistent, RK converges linearly in expectation. Since then, there is a large amount of work on the development of the Kaczmarz-type methods including accelerated RK methods \cite{liu2016accelerated,han2022pseudoinverse,loizou2020momentum,zeng2023adaptive}, randomized Douglas-Rachford methods \cite{Han2022-xh}, block Kaczmarz methods \cite{Nec19,needell2014paved,moorman2021randomized,Gow15}, greedy RK methods \cite{Bai18Gre,Gow19}, etc.
		As mentioned above, Zouzias and Freris \cite{Zou12} proposed the REK algorithm for the cases when the linear system $Ax = b$ is inconsistent.
		Subsequently, there is a large amount of work on the development of the REK-type methods, including its block or deterministic variants \cite{wu2021semiconvergence,needell2013two,DS2021,Du20Ran,wu2022two,wu2022extended,popa1998extensions,popa1999characterization,bai2019partially}, greedy randomized augmented Kaczmarz (GRAK) method \cite{bai2021greedy}, randomized extended Gauss-Seidel (REGS) method \cite{Du19,ma2015convergence}, etc. 
		
		Our particular interest lies in the randomized extended average block Kaczmarz (REABK) method  proposed by Du et al. \cite{Du20Ran}. Let $\left\{\mathcal{I}_1, \mathcal{I}_2, \dots, \mathcal{I}_\ell\right\}$ and $\left\{\mathcal{J}_1, \mathcal{J}_2, \dots, \mathcal{J}_t\right\}$ be partitions of $[m]$ and $[n]$, respectively. We use $A_{:, \mathcal{J}_{j_k}}$ and $A_{\mathcal{I}_{i_k}, :}$ to denote the column and row submatrix of $A$ indexed by $\mathcal{J}_{j_k}$ and $\mathcal{I}_{i_k}$, respectively. Inspired by the randomized average block Kaczmarz (RABK) method proposed by Necoara \cite{Nec19}, Du et al. \cite{Du20Ran}  introduced the following REABK method:
		\begin{equation} \label{REABK}
		\begin{aligned}
			z^{k+1}&=z^k-\frac{\alpha}{\Vert A_{:, \mathcal{J}_{j_k}} \Vert_F^2}A_{:, \mathcal{J}_{j_k}}A_{:, \mathcal{J}_{j_k}}^{\top} z^k, \\
			x^{k+1}&=x^k-\frac{\alpha}{\Vert A_{\mathcal{I}_{i_k}, :} \Vert_F^2}A_{\mathcal{I}_{i_k}, :}^\top(A_{\mathcal{I}_{i_k}, :}x^k-b_{\mathcal{I}_{i_k}}+z^{k+1}_{\mathcal{I}_{i_k}}),
		\end{aligned}
		\end{equation}
		where  $\alpha$ is the step-size. Particularly,	if the partition parameters $\ell=m$, $t=n$, and the step-size $\alpha=1$, then REABK reduces to the REK method \cite{Zou12}.
		The REABK method is able to be parallelized, but its effectiveness can be limited by the choice of step-size $\alpha$, which depends on the singular values of the submatrices involved \cite[Theorem 2.7]{Du20Ran}.
		Indeed, adopting an adaptive strategy for the step-size is often more beneficial in practical applications \cite{Nec19}. The authors provided an adaptive choice for the step-size of REABK in \cite{wu2022extended}. In this paper, we introduce a strategy for choosing adaptive step-sizes applicable to a broad range of stochastic extended iterative methods. Furthermore, our convergence analysis is more comprehensive, leading to an improved convergence factor.
		
		In addition, we note that the RK method can be seen as a variant of the SGD method \cite{Han2022-xh,needell2014stochasticMP,needell2014stochasticMP,ma2017stochastic,zeng2023randomized}.
		To solve the finite-sum problem $\min f(x):=\frac{1}{m}\sum_{i=1}^mf_i(x)$,
		SGD employs the following update rule
		\begin{equation}
		\label{SGD-iteration}
		x^{k+1}=x^k-\alpha_k\nabla f_{i_k}(x^k),
		\end{equation}
		where $\alpha_k$ is the step-size and $i_k$ is selected randomly. 
		If the objective function $f(x)=\frac{1}{2m}\|Ax-b\|^2_2=\frac{1}{2m}\sum_{i=1}^{m}(A_{i,:}x-b_i)^2$, then SGD reduces to the RK method.
		
		\subsubsection{Momentum acceleration}
		Building upon the success of the HBM method, various recent studies aim to extend this acceleration technique to the SGD method \cite{loizou2020momentum,barre2020complexity,sebbouh2021almost,han2022pseudoinverse,richtarik2020stochastic,
		loizou2021revisiting,morshed2020stochastic}.  The resulting method is  called the stochastic HBM (SHBM) method.
		Specifically, a heavy ball momentum term,  $\beta_{k}(x^k-x^{k-1})$, is incorporated into the SGD method \eqref{SGD-iteration}, leading to the iteration scheme
		%\begin{equation}
		%\label{SHBM-iteration}
		$$	x^{k+1}=x^k-\alpha_k\nabla f_{i_k}(x^k)+\beta_{k}(x^k-x^{k-1}).$$
		%\end{equation}
		However, the parameters $ \alpha_k$ and $\beta_k$ for this method may rely on certain problem parameters that are generally inaccessible. For instance, the choices of $\alpha_k$ and $\beta_k$  for the SHBM method \cite{loizou2020momentum,polyak1964some,ghadimi2015global,bollapragada2022fast} for solving the linear system $Ax=b$ require knowledge of the largest singular value as well as the smallest nonzero singular value of matrix $A$. Hence, it is an open problem whether one can  design a theoretically supported adaptive SHBM method \cite{barre2020complexity,bollapragada2022fast}.
		% . This has led to the open question of whether it is possible to design a theoretically supported adaptive SHBM method \cite{barre2020complexity,bollapragada2022fast}.
		For certain types of problems, some recent works have provided answers to this question \cite{saab2022adaptive,barre2020complexity,zeng2023adaptive,zeng2023fast,lorenz2023minimal,jin2024adaptive}.
		
		Our result is closely related to the work \cite{zeng2023adaptive}, where an adaptive SHBM (ASHBM) was developed to solve the stochastic problem reformulated from consistent linear systems.  
		It was demonstrated that ASHBM  exhibits convergence bound that is at least as that of the non-momentum method. 
		Actually, the strategy for adaptively updating the iterate $z^{k+1}$ in our stochastic extended iterative method with HBM is inspired by this ASHBM. 
		However, our adaptive strategy for updating $x^{k+1}$ is distinct from the aforementioned works.
		
		\subsection{Notations}
		
		For any random variable $\xi$, let $\mathbb{E}[\xi]$ denote its expectation. For any matrix $A \in \mathbb{R}^{m \times n}$, we use $A_{i, :}$, $A_{:, j}$, $A^\top$, $A^\dag$, $\|A\|_F$, $\sigma_{\max}(A)$, $\sigma_{\min}(A)$, $\text{rank}(A)$, $\text{Range}(A)$, and $\text{Range}(A)^{\perp}$ to denote the $i$-th row, the $j$-th column, the transpose, the Moore-Penrose pseudoinverse, the Frobenius norm, the largest singular value, the smallest nonzero singular value, the rank, the range space, and the orthogonal complement space of the range space of $A$, respectively.
		When $A$ is symmetric and positive semidefinite, we use $\lambda_{\max}(A)$ and $\lambda_{\min}(A)$ to denote the largest eigenvalue, and the smallest nonzero eigenvalue of $A$, respectively.
		Given $\mathcal{J} \subseteq [m]:=\{1,\ldots,m\}$, the complementary set of $\mathcal{J} $ is denoted by $\mathcal{J} ^c$, i.e. $\mathcal{J} ^c=[m] \setminus \mathcal{J}$.
		We use $A_{\mathcal{J} ,:}$ and $A_{:,\mathcal{J} }$ to denote the row and column submatrix indexed by $\mathcal{J} $, respectively.
		For any vector $b \in \mathbb{R}^m$, we use $b_i$, $\|b\|_2$, and $b_{\text{Range}(A)}$ to denote the $i$-th  entry, the Euclidean norm, and the orthogonal projection onto $\text{Range}(A)$ of $b$, respectively. The identity matrix and the unit vector are denoted by $I $
		and $e_i$, respectively.
		
		\subsection{Organization}
		
		The remainder of the paper is organized as follows. In Section \ref{Section-2}, we investigate stochastic reformulations of the least-squares problem.
		Section \ref{Section-3} and Section \ref{Section-4}  describe the stochastic extended iterative method and its momentum variant, respectively. Section \ref{Section-5} reports the mentioned numerical experiments and Section \ref{Section-6} concludes the paper. Proofs of all main results are provided in the appendix.
		
		\section{Stochastic reformulation of the least-squares problem}
		\label{Section-2}
		
		In this section, we will present several  equivalent  reformulations of the least-squares problem \eqref{LS_norm}. These reformulations serve three key purposes: (1) providing new theoretical perspectives on the least-squares problem, (2) establishing an accessible framework that enables researchers across different domains to incorporate their specialized knowledge, and (3) offering comprehensive insights into our proposed algorithm.  We incorporate  stochasticity through user-defined probability spaces $(\Omega, \mathcal{F}, P)$ and $(\overline{\Omega}, \overline{\mathcal{F}}, \overline{P})$, which characterize the ensembles of sampling matrices $S$ and $T$, respectively.
		
		{\bf Unconstrained stochastic optimization problem.}
		It is well-known that the least-squares problem \eqref{LS_norm}
		can be transformed equivalently into solving the following consistent augmented linear system \cite{arioli1989augmented, bai2021greedy}
		\begin{equation}\label{augmented}
		\left\{\begin{array}{ll}
			A^{\top}z=0,
			\\
			Ax=b-z,
		\end{array}
		\right.
		\end{equation}
		which can also be written in the form of the augmented matrix
		\[
		\begin{bmatrix}
		I & A\\
		A^\top &0
		\end{bmatrix}\begin{bmatrix}
		z\\
		x
		\end{bmatrix}=\begin{bmatrix}
		b\\
		0
		\end{bmatrix} 
		\Leftrightarrow \tilde{A} \tilde{x} = \tilde{b}.
		\]
		We now consider the following least-square problem for the augmented linear  system
		\[
		\min_{\tilde{x}\in \mathbb{R}^{m+n}} \Phi (\tilde{x}) := \frac{1}{2} \| \tilde{A} \tilde{x} - \tilde{b} \|_2^2. %= \frac{1}{2} \|A^{\top}z \|_2^2+\frac{1}{2} \|Ax-(b-z)\|^2_2 := %\Phi(x, z).
		\] Just as the sketching strategy applied in \cite{Gow15,zeng2023adaptive}, we take on a special sketching matrix $ 
		\tilde{S} = \begin{bmatrix}
		S & 0\\
		0 & T
		\end{bmatrix} $, where random matrices $S$ and $T$ are drawn from probability spaces $(\Omega, \mathcal{F}, P)$ and $(\overline{\Omega}, \overline{\mathcal{F}}, \overline{P})$, respectively. We can arrive at the following unconstrained stochastic optimization problem
		\begin{equation}
		\label{SOP1_re}
		\min\limits_{(x, z) \in \mathbb{R}^{n}\times \mathbb{R}^{m}} \Phi(x, z):=\mathbb{E} \left[\Phi_{S, T}(x, z)\right],
		\end{equation}
		where $
		\Phi_{S, T}(x, z):=\frac{1}{2} \|T^{\top}A^{\top}z \|_2^2+\frac{1}{2} \|S^\top(Ax-(b-z))\|^2_2.
		$

		{\bf Constrained stochastic optimization problem.}
		Let
		\begin{equation}
		\label{def_g_f}
		g_{T}(z):=\frac{1}{2} \|T^{\top}A^{\top}z\|_2^2, \ \ f_{S}(x, z):=\frac{1}{2} \|S^\top(Ax-(b-z))\|^2_2,
		\end{equation}
		and 
		\begin{equation}
		\label{def_Eg_f}
		g(z):=\mathbb{E}_{T \in \overline{\Omega}} \left[g_{T}(z)\right], \ \ f(x, z):=\mathbb{E}_{S\in\Omega} \left[f_{S}(x, z)\right].
		\end{equation} 
		We also consider the following constrained stochastic optimization problem
		\begin{equation}\label{SOP2}
		%	\begin{aligned}
			\text{Find} \ x, z \ \
			\text{subject to} \ g(z)=0,
			\ f(x, z)=0.
			%	\end{aligned}
			\end{equation}
			
			{\bf Multiobjective optimization problem.} Consider  the vector-valued function $\Psi_{S,T}: \mathbb{R}^{n}\times\mathbb{R}^m\to \mathbb{R}\times\mathbb{R}$ defined by $\Psi_{S,T}(x,z):=\left(g_T(z),f_S(x, z)\right)$. Let $\Psi(x,z):= \mathbb{E}[\Psi_{S,T}(x,z)]=\left(g(z),f(x, z)\right)$.  The associated unconstrained multiobjective optimization problem is defined as
			\begin{equation}\label{MOJ_xie}
		\min \limits_{(x, z)\in\mathbb{R}^{n} \times \mathbb{R}^{m}} \Psi(x,z).
		\end{equation}
		In most cases, it is impossible to find a single point that minimizes all objective functions at once, so it is necessary to consider the concept of \emph{Pareto} optimality.   We refer the reader to \cite{sawaragi1985theory,fukuda2014survey,tanabe2019proximal} for more details about  the multiobjective optimization problem. However, we can find a single point that minimizes all objective functions for \eqref{MOJ_xie}. %; see below.
		Indeed, since $g(z)\geq0 $ and $f(x,z)\geq0 $, and $g(A^{\dagger}b)=0$ and $f(A^{\dagger}b, b_{\text{Range}(A)^{\perp}})=0$, we know that $(A^{\dagger}b, b_{\text{Range}(A)^{\perp}})$ can minimize all objective functions for the multiobjective optimization problem \eqref{MOJ_xie}.
		
		\subsection{Equivalent expressions of the stochastic reformulations}
		
		For convenience, we use $\mathcal{X}_{\text{aug}}$, $\mathcal{X}_{\Phi}$, $\mathcal{X}_{g, f}$, and $\mathcal{X}_{\Psi}$ to denote the sets of minimizers for the augmented linear system \eqref{augmented}, the unconstrained stochastic optimization problem \eqref{SOP1_re}, the constrained stochastic optimization problem \eqref{SOP2}, and the multiobjective optimization problem \eqref{MOJ_xie}, respectively. 
		We have the following theorem which establishes the equivalence between these three stochastic formulations.
		
		\begin{theorem}\label{Iden}
		For any	probability spaces $(\Omega, \mathcal{F}, P)$ and $(\overline{\Omega}, \overline{\mathcal{F}}, \overline{P})$, we have $\mathcal{X}_{\Phi}=\mathcal{X}_{g, f}=$ $\mathcal{X}_{\Psi}$.
		\end{theorem}
		\begin{proof} It can be verified that $(A^{\dagger}b, b_{\text{Range}(A)^{\perp}})$ belongs to both $ \mathcal{X}_{\Phi}$ and $\mathcal{X}_{g, f}$, which implies that both of them are nonempty.
		Since $\Phi(A^{\dagger}b, b_{\text{Range}(A)^{\perp}})=0$, we know that 
		$$
		\left(\overline{x}, \overline{z}\right) \in \mathcal{X}_{\Phi} \Leftrightarrow\Phi\left(\overline{x}, \overline{z}\right)=0 \Leftrightarrow
		\left\{\begin{array}{ll}
			\mathbb{E}_{T \in \overline{\Omega}} \left[\frac{1}{2} \left\|T^{\top}A^{\top}\overline{z}\right\|_2^2\right]=0,
			\\
			\mathbb{E}_{S\in\Omega} \left[\frac{1}{2} \left\|S^\top(A\overline{x}-(b-\overline{z}))\right\|^2_2\right]=0,
		\end{array}
		\right.
		\Leftrightarrow
		\left\{\begin{array}{ll}
			g(\overline{z})=0,
			\\
			f(\overline{x}, \overline{z})=0,
		\end{array}
		\right.
		$$
		which implies that  $\mathcal{X}_{\Phi}=\mathcal{X}_{g, f}=$ $\mathcal{X}_{\Psi}$. This completes the proof of the theorem.
		\end{proof}
		
		\subsection{Exactness of the reformulations} In this subsection,
		we examine the \emph{exactness} of the reformulations. By exactness, we mean that  $x^*$ is a solution to \eqref{LS_norm} if and only if there exists a $z^*$ such that $(x^*,z^*)$ is a minimizer of the reformulated problem. We define 
		\begin{equation}
		\label{def_H_H}
		M:=\mathbb{E}_{S\in\Omega}[SS^\top], \ \ \overline{M}:=\mathbb{E}_{T\in\overline{\Omega}}[TT^\top],
		\end{equation}
		%H:=\mathbb{E}[SS^\top], \overline{H}:=\mathbb{E}[TT^\top]
		and
		\begin{equation}\label{def_setA}
		\mathcal{S}:=\{y\in \mathbb{R}^m \mid M^{\frac{1}{2}}y \in M^{\frac{1}{2}} \left(b+\text{Range}(A)\right),  y \in \text{Null}(A  \overline{M} A^{\top})\}.
		\end{equation}
		We have the following result which provides a  necessary and sufficient
		condition for the stochastic reformulations \eqref{SOP1_re} to be exact.
		
		\begin{theorem}\label{exactness}
		Suppose that $M$ and $\mathcal{S}$ are defined by \eqref{def_H_H} and \eqref{def_setA}, respectively. Then	$\mathcal{X}_{\Phi}=\mathcal{X}_{\text{aug}}$ if and only if $\mathcal{S}=b_{\text{Range}(A)^{\perp}}$ and $\text{Null}\left(A^{\top} M  A\right)=\text{Null}(A)$.
		\end{theorem}
		
		\begin{proof}
		Since $\mathcal{X}_{\text{aug}}=\{(x, z)  \mid A^{\top}z=0, Ax=b-z\}$, we know that for any $\left(\tilde{x}, \tilde{z}\right) \in \mathcal{X}_{\text{aug}}$, it holds that $\tilde{z} \in \text{Null}(A^{\top}) \cap \left(b +\text{Range}(A)\right)$. Note that 
		\begin{equation}\label{set_s_unique}
			\text{Null}(A^{\top}) \cap \left(b +\text{Range}(A)\right)=\text{Range}(A)^{\perp} \cap \left(b +\text{Range}(A)\right)=b_{\text{Range}(A)^{\perp}},
		\end{equation}
		which implies that  $\tilde{z}=b_{\text{Range}(A)^{\perp}}$. Hence,  $\mathcal{X}_{\text{aug}}$ can be rewritten as 
		$\mathcal{X}_{\text{aug}}=\{(x, b_{\text{Range}(A)^{\perp}})\mid Ax=b_{\text{Range}(A)}\}.$	Furthermore, according to the proof of Theorem \ref{Iden}, for any $\left(\overline{x}, \overline{z}\right) \in \mathcal{X}_{\Phi}$, it holds that
		\begin{equation} \nonumber
			\begin{aligned}
				(\overline{x}, \overline{z}) \in \mathcal{X}_{\Phi}
				&
				\Leftrightarrow \left\{\begin{array}{ll}
					g(\overline{z})=0
					\\
					f(\overline{x}, \overline{z})=0
				\end{array}
				\right.\Leftrightarrow
				\left\{\begin{array}{ll}
					\overline{M}^{\frac{1}{2}}A^{\top}\overline{z}=0
					\\
					M^{\frac{1}{2}}A\overline{x}=M^{\frac{1}{2}}(b-\overline{z})
				\end{array}
				\right. \Leftrightarrow
				\left\{\begin{array}{ll}
					\overline{z} \in \text{Null}\left(A \overline{M} A^{\top}\right),
					\\
					M^{\frac{1}{2}} \overline{z} \in M^{\frac{1}{2}} \left(b+\text{Range}(A)\right),
					\\
					M^{\frac{1}{2}}A\overline{x}=M^{\frac{1}{2}}(b-\overline{z}).
				\end{array}
				\right.
			\end{aligned}
		\end{equation}
		Consequently, we can get that $\mathcal{X}_{\Phi}=\mathcal{X}_{\text{aug}}$ if and only if
		\begin{equation}\nonumber
			\left\{\begin{array}{ll}
				\mathcal{S}=b_{\text{Range}(A)^{\perp}},
				\\
				\{x \in \mathbb{R}^n \mid Ax=b_{\text{Range}(A)}\}=\{x \in \mathbb{R}^n \mid M^{\frac{1}{2}}Ax=M^{\frac{1}{2}} b_{\text{Range}(A)}\},
			\end{array}
			\right.
		\end{equation}
		i.e.,  $\mathcal{S}=b_{\text{Range}(A)^{\perp}}$ and $\text{Null}\left(A^{\top} M  A\right)=\text{Null}(A)$.
		\end{proof}
		
		\begin{remark}
		\label{remark_exact}
		By	combining Theorems \ref{Iden} and \ref{exactness}, we can conclude that $\mathcal{X}_{\Phi}=\mathcal{X}_{g, f}=\mathcal{X}_{\Psi}=\mathcal{X}_{\text{aug}}$ if and only if $\mathcal{S}=b_{\text{Range}(A)^{\perp}}$ and $\text{Null}\left(A^{\top} M  A\right)=\text{Null}(A)$.
		\end{remark}
		From the proof of Theorem \ref{exactness}, we know that $x^*$ is a solution to the least-squares problem \eqref{LS_norm} if and only if $(x^*, b_{\text{Range}(A)^{\perp}})\in \mathcal{X}_{\text{aug}}$. Hence, Theorem \ref{exactness} present a necessary and sufficient condition for the reformulations to be exact.
		Next, we provide a sufficient condition for the exactness. 
		\begin{corollary}\label{suff}
		If $\overline{M}=\mathbb{E}_{T \in \overline{\Omega}} \left[TT^{\top}\right]$ and $M=\mathbb{E}_{S \in \Omega}\left[SS^\top\right]$ are positive definite matrices, then the stochastic  reformulations \eqref{SOP1_re}, \eqref{SOP2}, and \eqref{MOJ_xie}  are exact.
		\end{corollary}
		\begin{proof}
		Now we have $
		\mathcal{S} =\left\{y \in \mathbb{R}^m \mid y\in \text{Null}(A^\top), y\in b+\text{Range}(A) \right\}=b_{\text{Range}(A)^{\perp}}
		$, where the last equality follows from \eqref{set_s_unique}. Since $\text{Null}\left(A^{\top} M A\right)=\text{Null}(A)$,
		by Remark \ref{remark_exact}, we can conclude that this corollary holds.
		\end{proof}

		In this paper, we make the following assumption on the probability spaces $(\overline{\Omega}, \overline{\mathcal{F}}, \overline{P})$ and $(\Omega, \mathcal{F}, P)$.
		Corollary~\ref{suff} demonstrates that this assumption can provide a sufficient condition for the exactness of the stochastic reformulations \eqref{SOP1_re}, \eqref{SOP2}, and \eqref{MOJ_xie}. Furthermore, the finiteness requirement for the sampling space $\Omega$ is essential to guarantee that the key parameter $\Lambda_{\min}$ in \eqref{rho} is positive.

		\begin{assumption}
		\label{Ass}
		Let $(\overline{\Omega}, \overline{\mathcal{F}}, \overline{P})$ and $(\Omega, \mathcal{F}, P)$  be  the probability spaces from which the sampling matrices are drawn. We assume that $\mathbb{E}_{T \in \overline{\Omega}} \left[TT^{\top}\right]$ and $\mathop{\mathbb{E}}_{S\in\Omega} \left[S S^\top\right]$ are positive definite matrices. Additionally, we assume that the sampling space $\Omega$ is finite.
		\end{assumption}

		\section{Stochastic extended iterative methods with adaptive step-sizes}
		\label{Section-3}
		In this section, we consider employing the stochastic gradient descent (SGD) with adaptive step-sizes to solve the stochastic reformulations \eqref{SOP1_re}, \eqref{SOP2}, and \eqref{MOJ_xie} of the linear system \eqref{LS}.
		Specifically, let  $z^0\in b+\text{Range}(A)$ and $x^0\in \text{Range}(A^{\top})$ be  the initial points, our stochastic extended iterative method alternately minimizes $g(z)$ and $f(x,z)$ in \eqref{def_Eg_f} and the scheme can be written as 
		\begin{equation} \label{framework_a}
		\begin{aligned}
			z^{k+1}&=z^k-\mu_k\nabla g_{T_k} (z^k)=z^k-\mu_k AT_kT_k^{\top}A^{\top}z^k, \\
			x^{k+1}&=x^k-\alpha_k \nabla_x f_{S_k}(x^k,z^{k+1})=x^k-\alpha_k A^\top S_k S_k^\top(Ax^k-(b-z^{k+1})),
		\end{aligned}
		\end{equation}
		where $T_k$ and $S_k$ are drawn from the sample space $\overline{\Omega}$ and $\Omega$, respectively, and $\mu_k$ and $\alpha_k$ are the step-sizes.
		Specifically, the step-size $\mu_k$ is chosen as
		\begin{equation}\label{mu}
		\mu_k=
		\left\{\begin{array}{ll}
			(2-\eta) \overline{L}_{\text{adap}}^{(k)}, & \text{if} \; T_k^{\top}A^{\top}z^k \neq 0;
			\\
			0, & \text{otherwise},
		\end{array}
		\right.
		\end{equation}
		where
		\begin{equation}\label{overline_L_adapt}
		\begin{aligned}
			\overline{L}_{\text{adap}}^{(k)}:=\frac{\|T_k^{\top}A^{\top}z^k\|_2^2}{\|AT_kT_k^{\top}A^{\top}z^k\|_2^2},
		\end{aligned}
		\end{equation}
		and $\eta \in (0, 2)$ is the relaxation parameter for adjusting the step-size $\mu_k$. The step-size $\alpha_k$ is chosen as
		\begin{equation}\label{alpha_basic}
		\alpha_k=
		\left\{\begin{array}{ll}
			(2-\zeta) L_{\text{adap}}^{(k)}, & \text{if} \; S_k^{\top}\left(Ax^k-(b-z^{k+1})\right) \neq 0;
			\\
			0, & \text{otherwise},
		\end{array}
		\right.
		\end{equation}
		where
		\begin{equation}\label{L_adapt}
		\begin{aligned}
			L_{\text{adap}}^{(k)}:=\frac{\left\|S_k^{\top}\left(Ax^k-(b-z^{k+1})\right)\right\|_2^2}{\left\|A^{\top}S_kS_k^{\top}\left(Ax^k-(b-z^{k+1})\right)\right\|_2^2},
		\end{aligned}
		\end{equation}
		and $\zeta \in (0, 2)$ is the relaxation parameter for adjusting the step-size $\alpha_k$. 
		The following lemma ensures that the step-sizes $\mu_k$ and $\alpha_k$ are well-defined.
		\begin{lemma}[\cite{zeng2023adaptive}, Lemma 2.3]
		\label{lemma-non}
		Assume that $B\in\mathbb{R}^{p\times q}$ and the linear system $Bx=c$ is consistent. Then for any matrix $S \in \mathbb{R}^{p\times \ell}$ and any vector $\tilde{x} \in \mathbb{R}^{q}$, it holds that $B^\top SS^\top(B\tilde{x}-c) \neq 0$ if and only if $S^\top(B\tilde{x}-c) \neq 0$.
		\end{lemma}
		
		Since \(A^{\top}z = 0\) is consistent, Lemma \ref{lemma-non} implies that \(T_k^{\top}A^{\top}z^k \neq 0\) ensures \(AT_kT_k^{\top}A^{\top}z^k \neq 0\), confirming \(\overline{L}_{\text{adap}}^{(k)}\) is well-defined. By the iteration scheme of \(z^k\), we have \(z^k \in z^0 + \operatorname{Range}(A) = b + \operatorname{Range}(A)\), implying \(b - z^{k} \in \operatorname{Range}(A)\), so \(Ax = b - z^{k}\) is consistent. Thus, \(L_{\text{adap}}^{(k)}\) is also well-defined. Therefore, the step-sizes $\mu_k$ and $\alpha_k$ are well-defined.  We emphasize that when $T_k^{\top}A^{\top}z^k = 0$ then $AT_kT_k^{\top}A^{\top}z^k = 0$, and it holds that $z^{k+1}=z^k$ for any $\mu_k\in\mathbb{R}$. Hence, we set $\mu_k=0$. Similarly, when $S_k^{\top}\left(Ax^k-(b-z^{k+1})\right)=0$, we can set $\alpha_{k}=0$. 
		Now we are ready to formally state the stochastic extended iterative method with adaptive step-sizes in Algorithm \ref{ASE}. 
		
				\begin{algorithm}[htpb]
					\caption{Stochastic extended iterative method with adaptive step-sizes }
					\label{ASE}
					\begin{algorithmic}
						\Require
						$A\in \mathbb{R}^{m\times n}$, $b\in \mathbb{R}^m$, probability spaces $(\overline{\Omega}, \overline{\mathcal{F}}, \overline{P})$ and $(\Omega, \mathcal{F}, P)$, relaxation parameters $\eta \in (0, 2)$ and $\zeta \in (0, 2)$, $k=0$, and initial points $z^0\in b+\text{Range}(A)$ and $x^0\in \text{Range}(A^{\top})$.
						\begin{enumerate}
							\item[1:] Randomly select a sampling matrix $T_{k}\in \overline{\Omega}$.
							\item[2:] Compute the parameter $\mu_k$ in \eqref{mu} and then update
							$$
							z^{k+1}=z^k-\mu_k AT_kT_k^{\top}A^{\top}z^k.
							$$
							\item[3:] Randomly select a sampling matrix $S_{k}\in \Omega$.
							\item[4:] Compute the parameter $\alpha_k$ in \eqref{alpha_basic} and then
							update
							$$
							x^{k+1}=x^k-\alpha_k A^{\top} S_kS_k^{\top}(Ax^k-(b-z^{k+1})).
							$$
							\item[5:] If the stopping rule is satisfied, stop and go to output. Otherwise, set $k=k+1$ and go to Step $1$.
						\end{enumerate}
						
						\Ensure
						The approximate solution $x^k$.
					\end{algorithmic}
				\end{algorithm}
				
				\begin{remark} \label{AREABK}
				The REABK method \eqref{REABK} can be viewed as a special case of Algorithm \ref{ASE} with constant step-sizes.
				Recall that $\left\{\mathcal{I}_1, \mathcal{I}_2, \dots, \mathcal{I}_\ell\right\}$ and $\left\{\mathcal{J}_1, \mathcal{J}_2, \dots, \mathcal{J}_t\right\}$ are partitions of $[m]$ and $[n]$, respectively. Suppose that the sampling spaces $\Omega=\{I_{:, \mathcal{I}_{i}}/\Vert A_{\mathcal{I}_{i}, :} \Vert_F\}_{i=1}^\ell$, $\overline{\Omega}=\{I_{:, \mathcal{J}_{j}}/\Vert A_{:, \mathcal{J}_{j}} \Vert_F\}_{j=1}^t$, and the relaxation parameters $ \mu = \zeta = 1 $. Then, Algorithm \ref{ASE} has the  following iteration scheme
				\begin{equation} \label{aREABK}
					\begin{aligned}
						z^{k+1}&=z^k-\frac{\mu_k}{\Vert A_{:, \mathcal{J}_{j_k}} \Vert_F^2}A_{:, \mathcal{J}_{j_k}}A_{:, \mathcal{J}_{j_k}}^{\top} z^k, \\
						x^{k+1}&=x^k-\frac{\alpha_k}{\Vert A_{\mathcal{I}_{i_k}, :} \Vert_F^2}A_{\mathcal{I}_{i_k}, :}^\top(A_{\mathcal{I}_{i_k}, :}x^k-(b-z^{k+1})_{\mathcal{I}_{i_k}}),
					\end{aligned}
				\end{equation}
				where the step-sizes $\mu_k$ and $\alpha_k$ are defined as follows
				$$
				\mu_k=
				\left\{\begin{array}{ll}
					%(2-\eta) 
					\frac{ \Vert A_{:, \mathcal{J}_{j_k}}^{\top}z^k \Vert_2^2 \Vert A_{:, \mathcal{J}_{j_k}}\Vert_F^2}{\Vert A_{:, \mathcal{J}_{j_k}} A_{:, \mathcal{J}_{j_k}}^{\top}z^k \Vert_2^2}, & \text{if} \;\; A_{:, \mathcal{J}_{j_k}}^{\top}z^k \neq 0, \vspace{1ex}
					\\
					0, & otherwise,
				\end{array}
				\right.
				$$
				and
				$$
				\alpha_k=
				\left\{\begin{array}{ll}
					%(2-\zeta) 
					\frac{\Vert A_{\mathcal{I}_{i_k}, :}x^k-(b-z^{k+1})_{\mathcal{I}_{i_k}} \Vert_2^2 \Vert A_{\mathcal{I}_{i_k}, :} \Vert_F^2}{\|A_{\mathcal{I}_{i_k}, :}^{\top} \left(A_{\mathcal{I}_{i_k}, :}x^k-(b-z^{k+1})_{\mathcal{I}_{i_k}}\right)\|_2^2}, & \text{if} \; A_{\mathcal{I}_{i_k}, :}x^k-(b-z^{k+1})_{\mathcal{I}_{i_k}} \neq 0, \vspace{1ex}
					\\
					0, & otherwise.
				\end{array}
				\right.
				$$
				It is evident that Algorithm \ref{ASE}  and the REABK method \eqref{REABK} share a similar algorithmic framework. Specifically, when the parameters  $\mu_k$ and $\alpha_k$ are set to be the same constant step-size $\alpha$, Algorithm \ref{ASE} reduces to the REABK method.
				Furthermore,  the randomized multiple row method proposed in \cite{wu2022extended}, where adaptive step-sizes are introduced for REABK, can also be viewed as a special case of Algorithm \ref{ASE}.
			\end{remark}

			\subsection{Convergence analysis}
			For convenience, we define
			\begin{equation}\label{matrix-H-}
				H:=\mathop{\mathbb{E}}\limits_{S \in \Omega}\left[\frac{SS^{\top}}{\left\|A^{\top} S\right\|_2^2}\right] \ \ \text{and} \ \ \overline{H}:=\mathop{\mathbb{E}}\limits_{T \in \overline{\Omega}}\left[\frac{TT^{\top}}{\|AT\|_2^2}\right]. 
			\end{equation}
			The following lemma shows that both $H$ and $\overline{H}$ are positive definite under Assumption \ref{Ass}.
			\begin{lemma}[\cite{lorenz2023minimal}, Lemma 2.3]
				\label{lemma_}
				Suppose Assumption \ref{Ass} holds. Then, the matrices $H$ and $\overline{H}$ are positive definite.
			\end{lemma}
			Furthermore, we define
			\begin{equation}\label{rho}
				\rho_{z}:=1-\eta(2-\eta) \sigma_{\min}^2(\overline{H}^{\frac{1}{2}}A^{\top}), 
				\ \ \text{and} \ \ 
				\Lambda_{\min}:=\inf\limits_{S\in \Omega, A^{\top}S\neq 0} \lambda_{\min}\left(\frac{A^\top SS^\top A}{\|A^{\top}S\|_2^2}\right),
				%\label{lambda-min}
			\end{equation}
			where $\sigma_{\min}(\cdot)$ and $\lambda_{\min}(\cdot)$ denote the smallest nonzero singular value and the smallest nonzero eigenvalue of a given matrix, respectively. Since the sampling space \(\Omega\) is finite (as assumed in Assumption \ref{Ass}), it is ensured that \(\Lambda_{\min}>0\).
			
			For any $0 < \epsilon \leq 1$, we set
			\begin{equation}\label{c}
				c_{\zeta, \epsilon}:=
				\left\{\begin{array}{ll}
					(2-\zeta) \left( 1-(1+\epsilon)(1-\zeta)\right) , \;\;\,\ \text{if} \; \zeta \in (0, 1);
					\\
					\\
					(2-\zeta)\left( 1+(1-\epsilon)(\zeta-1)\right), \;\;\,\ \text{if} \; \zeta \in [1, 2),
				\end{array}
				\right.
			\end{equation}
			\begin{equation}\label{d}
				d_{\zeta, \epsilon}:=
				\left\{\begin{array}{ll}
					(2-\zeta) \left( 1+(\epsilon^{-1}+1)(1-\zeta)\right), \;\;\,\ \text{if} \; \zeta \in (0, 1);
					\\
					\\
					(2-\zeta)\left( 1+(\epsilon^{-1}-1)(\zeta-1)\right), \;\;\,\ \text{if} \; \zeta \in [1, 2),
				\end{array}
				\right.
			\end{equation}
			\begin{equation}\label{v}
				\rho_{x}:=1-c_{\zeta, \epsilon} \sigma_{\min}^2\left(H^{\frac{1}{2}}A\right),
				%\end{equation}
				\ \	\text{and} \ \
				%\begin{equation}\label{rho_full}
				\rho:=\max\{\rho_{z}, \rho_{x}\}.
			\end{equation}
			We have the following convergence result for Algorithm \ref{ASE}.
			\begin{theorem}
				\label{x-ASEM-convergence}
				Suppose the probability spaces $(\overline{\Omega}, \overline{\mathcal{F}}, \overline{P})$ and $(\Omega, \mathcal{F}, P)$ satisfy Assumption \ref{Ass}. 
				For any given linear system $Ax=b$, let  $\{x^k\}_{k\geq0}$ be the iteration sequence generated by Algorithm \ref{ASE}. Then, if $\zeta \in (0, \frac{1}{2}]$ and $\epsilon \in (0, \frac{\zeta}{1-\zeta})$, or $\zeta \in (\frac{1}{2}, 2)$ and $\epsilon \in (0, 1]$, we have
				$$
				\mathbb{E}\left[\|x^k-A^{\dagger}b\|_2^2\right] \leq 
				\rho^{k} \left(\|x^0-A^{\dagger}b\|_2^2+\frac{d_{\zeta, \epsilon} \lambda_{\max}(H)}{\Lambda_{\min}\max\left\{\left|1-\frac{\rho_{x}}{\rho_{z}}\right|, \frac{1}{k}\right\}} \|z^0-b_{\text{Range}(A)^{\bot}}\|_2^2\right), 
				$$
				where $H$, $\rho_{z}$, $\Lambda_{\min}$, $d_{\zeta, \epsilon}$, $\rho_{x}$, and $\rho$ are given by \eqref{matrix-H-}, \eqref{rho},  \eqref{d}, and \eqref{v}, respectively.
			\end{theorem}
			
			Based on Theorem \ref{x-ASEM-convergence}, we have the following corollary.
			
			\begin{corollary} \label{corollary}
				Under the same conditions of Theorem \ref{x-ASEM-convergence}, the iteration sequence $\{x^k\}_{k \geq 0}$  generated by Algorithm \ref{ASE} with relaxation parameters $\eta=\zeta=1$ satisfies
				$$
				\begin{aligned}
					\mathbb{E}[\|x^k-A^{\dagger}b\|_2^2] \leq  
					\rho^k \left(\|x^0-A^{\dagger}b\|_2^2+\frac{ \lambda_{\max}(H)}{\Lambda_{\min} \Gamma_k } \|z^0-b_{\text{Range}(A)^{\bot}}\|_2^2\right),
				\end{aligned}
				$$
				where  $\overline{H}$, $H$, $\Lambda_{\min}$ are defined as  \eqref{matrix-H-} and \eqref{rho}, respectively,  $
				\Gamma_k:=\max\{|\sigma_{\min}^2(\overline{H}^{\frac{1}{2}}A^{\top})-\sigma_{\min}^2(H^{\frac{1}{2}}A)|\big/
				(1-\sigma_{\min}^2(\overline{H}^{\frac{1}{2}}A^{\top})), 1/k\}
				$, and 
				$
				\rho=1-\min\{\sigma_{\min}^2(\overline{H}^{\frac{1}{2}}A^{\top}), \sigma_{\min}^2(H^{\frac{1}{2}}A)\}
				$.
			\end{corollary}
			
			As Algorithm \ref{ASE} can recover many existing methods.  We here discuss the relationship between Corollary \ref{corollary} and the convergence results of these methods.

			\begin{remark}[REK] \label{rek-iteration}
				%We compare the upper bound in Corollary \ref{corollary} with that of the REK method. 
				When $\overline{\Omega}=\left\{ e_j/\Vert A_{:, j} \Vert_2 \right\}_{j=1}^n$ and $T_k=e_j/\Vert A_{:, j} \Vert_2$ selected with probability $\Vert A_{:,j} \Vert_2^2/\Vert A \Vert_F^2$, and $\Omega = \left\{ e_i/\Vert A_{i, :} \Vert_2 \right\}_{i=1}^m $ and $S_k=e_i/\Vert A_{i, :} \Vert_2$ selected with probability $\Vert A_{i, :} \Vert_2^2/\Vert A \Vert_F^2$, Algorithm \ref{ASE} with $\eta=\zeta=1$ reduces to the REK method.
				Now that $\overline{H}=I/\Vert A \Vert_F^2$, $H=I/\Vert A \Vert_F^2$, and $\Lambda_{\min}=1$, it then follows from Corollary \ref{corollary} that
				$$
				\mathop{\mathbb{E}} [\| x^k-A^\dagger b \|_2^2 ] \leq \left(1-\frac{\sigma_{\min}^{2} (A)}{\Vert A \Vert_F^2 }\right)^k \left(\| x^0-A^\dagger b \|_2^2+k \frac{\|z^0-b_{\text{Range}(A)^{\bot}}\|_2^2}{\Vert A \Vert_F^2} \right),
				$$
				which coincides with the tight convergence result of the REK method obtain in \cite{Du19}.
			\end{remark}
			
			\begin{remark}[REABK]
				\label{remark_sampling}
				Under the conditions in Remark \ref{AREABK}, when 
				the sampling matrices $T_k= I_{:, \mathcal{J}_{j}}/\Vert A_{:, \mathcal{J}_{j}} \Vert_F$  and $S_k= I_{:, \mathcal{I}_{i}}/\Vert A_{\mathcal{I}_{i}, :} \Vert_F$ are selected with probability $\Vert A_{:, \mathcal{J}_{j}} \Vert_F^2/\Vert A \Vert_F^2$ and $\Vert A_{\mathcal{I}_{i}, :} \Vert_F^2/\Vert A \Vert_F^2$ respectively, the authors \cite[Theorem 3.1]{wu2022extended} showed that for the REABK method with adaptive step-sizes, it holds that
				\begin{equation}\label{ERMR_con}
					\mathbb{E}[\|x^k-A^{\dagger}b\|_2^2]
					\leq
					\rho_3^k \left(\|x^0-A^{\dagger}b\|_2^2+\frac{\ell}{|\rho_1-\rho_2| \Gamma_{\min}^{\mathcal{I}} \|A\|_F^2} \|z^0-b_{\text{Range}(A)^{\bot}}\|_2^2\right),
				\end{equation}
				where $\rho_{1}:=1-\frac{1}{\Gamma_{\max}^{\mathcal{I}}}\frac{\sigma_{\min}^2(A)}{\|A\|_F^2}$, $\rho_{2}:=1-\frac{1}{\Gamma_{\max}^{\mathcal{J}}}\frac{\sigma_{\min}^2(A)}{\|A\|_F^2}$, $\rho_{3}:=\max\{\rho_{1}, \rho_{2}\}$, $\Gamma_{\min}^{\mathcal{I}}:=\mathop{\min}\limits_{i \in [\ell]}\left\{\frac{\sigma_{\min}^2(A_{\mathcal{I}_i, :})}{\Vert A_{\mathcal{I}_i, :} \Vert_F^2}\right\}$, $\Gamma_{\max}^{\mathcal{I}}:=\mathop{\max}\limits_{i \in [\ell]} \left\{\frac{\sigma_{\max}^2(A_{\mathcal{I}_i, :})}{\Vert A_{\mathcal{I}_i, :} \Vert_F^2}\right\}$, $\Gamma_{\max}^{\mathcal{J}}:=\mathop{\max}\limits_{j \in [t]} \left\{\frac{\sigma_{\max}^2(A_{:, \mathcal{J}_j})}{\Vert A_{:, \mathcal{J}_j} \Vert_F^2}\right\}$, and $\ell$ is the number of elements in $ \Omega $.
				It is necessary to assume that $\rho_{1}\neq\rho_{2}$ to ensure well-definedness of the convergence result \eqref{ERMR_con}.

				Meanwhile, the parameters in Corollary \ref{corollary} are specified as 
				$\Lambda_{\min}=\mathop{\min}\limits_{i \in [\ell]}\left\{\frac{\sigma_{\min}^2(A_{\mathcal{I}_i, :})}{\Vert A_{\mathcal{I}_i, :} \Vert_2^2}\right\}$,
				$$
				\rho_{z}=1-\sigma_{\min}^2\left(\overline{H}^{\frac{1}{2}}A^{\top}\right)=1-\frac{1}{\|A\|_F^2} \sigma_{\min}^2\left(\sum\limits_{j \in [t]} \frac{\Vert A_{:, \mathcal{J}_j} \Vert_F}{\Vert A_{:, \mathcal{J}_j} \Vert_2} I_{:, \mathcal{J}_j} A_{:, \mathcal{J}_j}^{\top}\right) \leq 1-\frac{\sigma_{\min}^2 (A)}{\Gamma_{\max}^{\mathcal{J}} \Vert A \Vert_F^2}=\rho_2,
				$$
				$$
				\rho_{x}=1-\sigma_{\min}^2\left(H^{\frac{1}{2}}A^{\top}\right)=1-\frac{1}{\|A\|_F^2} \sigma_{\min}^2\left(\sum\limits_{i \in [\ell]} \frac{\Vert A_{\mathcal{I}_i, :} \Vert_F}{\Vert A_{\mathcal{I}_i, :} \Vert_2} I_{:, \mathcal{I}_i} A_{\mathcal{I}_i, :}\right) \leq 1-\frac{\sigma_{\min}^2 (A)}{\Gamma_{\max}^{\mathcal{I}} \Vert A \Vert_F^2}=\rho_1,
				$$
				and $\rho=\max\{\rho_{z}, \rho_{x}\}$, where $\overline{H}=\frac{1}{\|A\|_F^2} \sum\limits_{j \in [t]} \frac{\Vert A_{:, \mathcal{J}_j} \Vert_F^2}{\Vert A_{:, \mathcal{J}_j} \Vert_2^2} I_{:, \mathcal{J}_j} I_{:, \mathcal{J}_j}^{\top}$ and $H=\frac{1}{\|A\|_F^2} \sum\limits_{i \in [\ell]} \frac{\Vert A_{\mathcal{I}_i, :} \Vert_F^2}{\Vert A_{\mathcal{I}_i, :} \Vert_2^2} I_{:, \mathcal{I}_i} I_{:, \mathcal{I}_i}^{\top}$.
				Let $\psi_{\max}:=\mathop{\max}\limits_{i \in [\ell]} \left\{\Vert A_{\mathcal{I}_i, :} \Vert_F^2/\Vert A_{\mathcal{I}_i, :} \Vert_2^2\right\}$, 
				Corollary \ref{corollary} then indicates that the REABK method with adaptive step-sizes %as discussed in Remark \ref{AREABK}
				shares the following convergence result
				$$
				\mathbb{E}\left[\|x^k-A^{\dagger}b\|_2^2\right]
				\leq
				\rho^k \left(\|x^0-A^{\dagger}b\|_2^2+\frac{\psi_{\max}}{\mathop{\max}\{|1-\frac{\rho_x}{\rho_z}|, \frac{1}{k}\} \Lambda_{\min} \|A\|_F^2} \|z^0-b_{\text{Range}(A)^{\bot}}\|_2^2\right), 	$$
				which is well-defined for any $\rho_z, \rho_x \in (0, 1)$.
				%						\textcolor{red}{In addition, since $\rho \leq \rho_{3}$, we can infer that our convergence factor is better than that in  \eqref{ERMR_con}.}
				In addition, since $\rho_{z} \leq \rho_2$ and $\rho_{x} \leq \rho_1$, it follows that $\rho=\max\{\rho_{z}, \rho_{x}\} \leq \max\{\rho_{1}, \rho_{2}\}=\rho_{3}$. This implies that our convergence factor is better than that in  \eqref{ERMR_con}.
			\end{remark}

			\section{Stochastic extended iterative method with adaptive HBM }
			\label{Section-4}
			
			This section aims to enhance the stochastic extended iterative method by incorporating adaptive HBM. Building on the framework established in \eqref{framework_a}, we propose the following adaptive HBM-based iteration scheme:
			\begin{equation} \label{AmSEIM}
				\begin{aligned}
					z^{k+1}&=z^k-\mu_k AT_kT_k^{\top}A^{\top}z^k+\omega_k(z^k-z^{k-1}), \\
					x^{k+1}&=x^k-\alpha_k A^\top S_k S_k^\top(Ax^k-(b-z^{k+1}))+\beta_k(x^k-x^{k-1}),
				\end{aligned}
			\end{equation}
			where $\mu_k$ and $\alpha_k$ are the step-sizes, and $\omega_k$ and $\beta_k$ are the momentum parameters.
			We choose $z^0 \in b+\text{Range}(A)$  and $x^0 \in \text{Range}(A^{\top})$ as initial points, and  then employ Algorithm \ref{ASE} with the parameters 
			$\eta=\zeta=1$ to generate $z^1$ and $x^1$. We now focus on adaptive strategies for updating the parameters $\mu_k,\omega_k$ and $\alpha_k,\beta_k$, respectively.
			
			\subsection{Updating the parameters $\mu_k$ and $\omega_k$ } 
			\label{SHBM_z}
			
			The adaptive strategy for updating the parameters $\mu_k$ and $\omega_k$ is  inspired 
			by the recent work \cite{zeng2023adaptive}. For convenience, we define $ p^k:=\nabla g_{T_k} (z^k)=AT_kT_k^{\top}A^{\top}z^k$. If $z^k-z^{k-1}$ is parallel to $p^k$, i.e. $\| p^k \|_2^2 \| z^k-z^{k-1} \|_2^2 - \langle p^k, z^k-z^{k-1} \rangle^2=0$, we choose $\mu_k$ as \eqref{mu} with $\eta=1$ and set $\omega_k=0$. %If $z^k-z^{k-1}$ is not parallel to $p^k$, 
			Otherwise, we expect to choose $\mu_k$ and $\omega_k$ that minimize the error $\|z^{k+1}-b_{\text{Range}(A)^{\bot}}\|_2$. 
			In this case, the optimal values of $\mu_k$ and $\omega_k$ can be expressed as
			\begin{equation}\label{mu_omega_HB1}
				\left\{
				\begin{array}{ll}
					\mu_k^{\operatorname{opt}}=\frac{\| z^k-z^{k-1} \|_2^2 \|T_k^{\top}A^{\top}z^k\|_2^2}{\| p^k \|_2^2 \| z^k-z^{k-1} \|_2^2 - \langle p^k, z^k-z^{k-1} \rangle^2}, \\[0.5cm]
					\omega_k^{\operatorname{opt}}=\frac{\langle p^k, z^k-z^{k-1} \rangle \|T_k^{\top}A^{\top}z^k\|_2^2}{\| p^k \|_2^2 \| z^k-z^{k-1} \|_2^2 - \langle p^k, z^k-z^{k-1} \rangle^2}.
				\end{array}
				\right.
			\end{equation}
			One may refer to  \cite[Section 4]{zeng2023adaptive} for more details.
			The procedure of updating the vector $z^{k+1}$ is formally described in Stage I.
				
				\begin{table}[htpb]
					\centering
					{
						\begin{tabular}{  |l|  }
							\hline
							\qquad \qquad \qquad \qquad \qquad \qquad \quad \qquad \textbf{Stage I} \qquad \qquad \qquad \qquad \qquad \qquad \quad  \qquad\\
							1: Randomly select a sampling matrix $T_{k}\in \overline{\Omega}$. \\
							2: Set $p^k=AT_kT_k^{\top}A^{\top}z^k$.\\
							3: If $\| p^k \|_2^2 \| z^k-z^{k-1} \|_2^2 - \langle p^k, z^k-z^{k-1} \rangle^2 \neq 0$ \\
							
							\qquad\qquad Compute $\mu_k=\mu_k^{\operatorname{opt}}$ and $\omega_k=\omega_k^{\operatorname{opt}}$ by \eqref{mu_omega_HB1}. \\
							
							\qquad	Otherwise \\
							
							\qquad\qquad Compute $\mu_k$ by \eqref{mu} with $\eta=1$ and set $\omega_k=0$. \\
							
							\quad \, End if \\
							4: Update $z^{k+1}=z^k-\mu_k p^k+\omega_k(z^k-z^{k-1})$. \\
							\hline
						\end{tabular}
					}
				\end{table}
				
				\subsection{Updating the parameters $\alpha_k$ and $\beta_k$ } 
			%For clarity and conciseness, 
			Similarly, we define $q^k:=\nabla_x f_{S_k}(x^k,z^{k+1})=A^{\top}S_k S_k^{\top}(Ax^k-(b-z^{k+1}))$. If $x^k-x^{k-1}$ and $q^k$ are linearly dependent, i.e. $\| q^k \|_2^2 \| x^k-x^{k-1} \|_2^2 - \langle q^k, x^k-x^{k-1} \rangle^2=0$, we determine $\alpha_k$ using \eqref{alpha_basic} with $\zeta=1$ and set $\beta_k=0$. We mainly discuss the scenario where  $x^k-x^{k-1}$ and $q^k$ are linearly independent.
			
			Ideally, one might want to determine $\alpha_k$ and $\beta_k$ such that  $\|x^{k+1}-A^{\dagger}b\|_2$ is minimized.
			However, this may be unattainable in practice as the linear system is not necessarily consistent.
			Note that $A^{\dagger}(b-z^{k+1})$ can be regarded as an estimation of $A^{\dagger}b$ as $z^{k+1}$ approximates $b_{\text{Range}(A)^{\bot}}$ (See Lemma \ref{z-ASHBM-convergence}). 
			Consequently, we shift to considering choosing $\alpha_k$ and $\beta_k$ such that  $\|x^{k+1}-A^{\dagger}(b-z^{k+1})\|_2$ is minimized, i.e.
			
			\begin{equation}\label{opt-prob}
				\begin{aligned}
					\min\limits_{ \alpha,\beta\in\mathbb{R}}& \ \ \|x-A^{\dagger}(b-z^{k+1})\|_2^2\\
					\text{subject to}& \ \ x=x^{k}-\alpha A^{\top}S_kS_k^{\top}(Ax^k-(b-z^{k+1}))+\beta(x^{k}-x^{k-1}).
				\end{aligned}
			\end{equation}
			Let $\mathcal{H}_k:=x^k+\text{Span}\{q^k, x^k-x^{k-1}\}$ denote the feasible set of the problem above. We present an intuitive geometric explanation of this strategy in Figure \ref{GI1}.
			The minimizers of \eqref{opt-prob} can be expressed as
			\begin{equation}\label{Adap}
				\left\{
				\begin{array}{ll}
					\alpha_k^{\operatorname{opt}}=\frac{\Vert u^k \Vert_2^2 \Vert x^k-x^{k-1} \Vert_2^2-\langle q^k, x^k-x^{k-1}\rangle \langle x^k-x^{k-1}, x^k-A^{\dagger}(b-z^{k+1}) \rangle}{\Vert q^k \Vert_2^2 \Vert x^k-x^{k-1} \Vert_2^2-\langle q^k,  x^k-x^{k-1} \rangle^2}, \\[0.5cm]
					\beta_k^{\operatorname{opt}}=\frac{\Vert u^k \Vert_2^2 \langle q^k, x^k-x^{k-1} \rangle- \Vert q^k \Vert_2^2 \langle x^k-x^{k-1}, x^k-A^{\dagger}(b-z^{k+1}) \rangle}{\Vert q^k \Vert_2^2 \Vert x^k-x^{k-1} \Vert_2^2-\left\langle q^k,  x^k-x^{k-1} \right \rangle^2},
				\end{array}
				\right.
			\end{equation}
			where  $u^k:=S_k^{\top}(Ax^k-(b-z^{k+1}))$.
			While the term $\langle x^k-x^{k-1}, x^k-A^{\dagger}(b-z^{k+1}) \rangle$ 
			seems to still require the knowledge of the unknown vector $A^{\dagger}b$, it can be computed using intermediate variables through iteration. 
			\begin{figure}[hptb]
				\centering
				\begin{tikzpicture}
					\draw (0,0)--(5,0)--(6,2)--(1,2)--(0,0);
					
					\filldraw (1.5,0.7) circle [radius=1pt]
					(4,3.5) circle [radius=1pt]
					(4,1.2) circle [radius=1pt];
					\draw (1.5,0.4) node {$x^k$};
					\draw (5.2,3.5) node {$A^{\dagger}(b-z^{k+1})$};
					\draw (4.45,1.15) node {$x^{k+1}$};
					\draw (4.8,0.3) node {$\mathcal{H}_k$};
					\draw [dashed,-stealth] (4,3.5) -- (4,1.25);
					\draw [-stealth] (1.5,0.7) -- (3.95,1.19);
				\end{tikzpicture}
				\caption{A geometric interpretation of our design. The next iterate $x^{k+1}$ arises such that $x^{k+1}$ is the orthogonal projection of $A^{\dagger}(b-z^{k+1})$ onto the affine set $\mathcal{H}_k=x^k+\text{Span}\{q^k, x^k-x^{k-1}\}$.}
				\label{GI1}
			\end{figure}
			
			Let  $h^1:=-\alpha_0 S_0 S_0^{\top}\left(Ax^0-(b-z^1)\right)$, and define $h^i$ ($i \geq 2$) as
			\begin{equation}
				\label{h}
				\begin{aligned}
					h^i:=-\alpha_{i-1} S_{i-1}u^{i-1}+\beta_{i-1} h^{i-1},
				\end{aligned}
			\end{equation}
			so that $x^i-x^{i-1}=A^{\top} h^{i}$ for all $i \geq 1$. In addition, according to the design of the sequence $\{x^{k}\}_{k \geq 0}$, it can be verified that $\langle x^i-x^{i-1}, x^i-A^{\dagger}(b-z^i) \rangle=0$ for all $i \geq 1$. Hence, we can get
			
			$$
			\begin{aligned}
				\langle x^k-x^{k-1}, x^k-A^{\dagger}(b-z^{k+1}) \rangle
				&=\langle x^k-x^{k-1}, A^{\dagger}(z^{k+1}-z^k) \rangle\\
				&=\langle h^{k}, AA^{\dagger}(z^{k+1}-z^k) \rangle
				=\langle h^{k}, z^{k+1}-z^k \rangle,
			\end{aligned}
			$$
			where the last equality follows from the fact that $z^{k+1}-z^k \in \text{Range}(A)$. As a result, \eqref{Adap} can be computed by
			\begin{equation}\label{Adap1}
				\left\{
				\begin{array}{ll}
					\alpha_k^{\operatorname{opt}}=\frac{\Vert u^k \Vert_2^2 \Vert x^k-x^{k-1} \Vert_2^2-\langle q^k, x^k-x^{k-1}\rangle \langle h^{k}, z^{k+1}-z^k \rangle}{\Vert q^k \Vert_2^2 \Vert x^k-x^{k-1} \Vert_2^2-\langle q^k,  x^k-x^{k-1} \rangle^2}, \\[0.5cm]
					\beta_k^{\operatorname{opt}}=\frac{\Vert u^k \Vert_2^2 \langle q^k, x^k-x^{k-1} \rangle- \Vert q^k \Vert_2^2 \langle h^{k}, z^{k+1}-z^k \rangle}{\Vert q^k \Vert_2^2 \Vert x^k-x^{k-1} \Vert_2^2-\left\langle q^k,  x^k-x^{k-1} \right \rangle^2}.
				\end{array}
				\right.
			\end{equation}
			
			The procedure of updating the iterate $x^{k+1}$ is formally described in Stage II. Now, we have already constructed the  stochastic extended iterative method with HBM described in Algorithm \ref{ASEHBM}.
				
				\begin{table}[htpb]
					\centering
					{
						\begin{tabular}{  |l|  }
							\hline
							\qquad \qquad \qquad \qquad \qquad \qquad \qquad \textbf{Stage II}  \qquad \qquad \qquad \qquad \qquad \qquad \qquad\\
							1: Randomly select a sampling matrix $S_{k}\in \Omega$. \\
							2:  Set $u^k=S_k^{\top}(Ax^k-(b-z^{k+1}))$ and $q^k=A^{\top}S_k u^k$.\\
							3:
							%Define $u^k=S_k^{\top} (Ax^k-(b-z^{k+1}))$ and $q^k=A^{\top}S_k u^k$. \\
							
							If \text{dim}($\Pi_k$)=2 \\
							
							\;\;\; \qquad Set $\alpha_k=\alpha_k^{opt}$ and $\beta_k=\beta_k^{opt}$. \\
							
							\;\;\; Otherwise \\
							\;\;\; \qquad Compute $\alpha_k$ in \eqref{alpha_basic} with $\zeta=1$ and set $\beta_{k}=0$. \\
							
							\;\;\; End if \\
							4: Update
							\qquad $x^{k+1}=x^k-\alpha_k q^k+\beta_k (x^k-x^{k-1})$, \\
							\qquad  \qquad \qquad \;\; $h^{k+1}=-\alpha_k S_k u^k+\beta_k h^{k}$. \\
							\hline
						\end{tabular}
					}
				\end{table}
				
				\begin{algorithm}[htpb]
					\renewcommand{\thealgorithm}{2}
					\caption{ Stochastic extended iterative method with adaptive HBM} \label{ASEHBM}
					\begin{algorithmic}
						\Require
						$A\in \mathbb{R}^{m\times n}$, $b\in \mathbb{R}^m$, probability spaces $(\overline{\Omega}, \overline{\mathcal{F}}, \overline{P})$ and $(\Omega, \mathcal{F}, P)$, $k=1$, and initial points $z^0\in b+\text{Range}(A)$ and $x^0\in \text{Range}(A^{\top})$.
						\begin{enumerate}
							\item[1:] Update $z^1, x^1$ by Algorithm \ref{ASE} with $\eta=\zeta=1$, and set $h^1=-\alpha_0 S_0 S_0^{\top}\left(Ax^0-(b-z^1)\right)$.
							\item[2:] Update $z^{k+1}$ by Stage I.
							\item[3:] Update $x^{k+1}$ and $h^{k+1}$ by Stage II.
							\item[4:] If the stopping rule is satisfied, stop and go to output. Otherwise, set $k=k+1$ and go to Step $2$.
						\end{enumerate}
						
						\Ensure
						The approximate solution $x^k$.
					\end{algorithmic}
				\end{algorithm}
				
				\subsection{Convergence analysis}
			
			We first introduce some auxiliary variables. Denote
			\begin{equation}\label{rho_z}
				\hat{\rho}_z:=1-\sigma_{\min}^2(\overline{H}^{\frac{1}{2}}A^{\top}), \quad \hat{\rho}_x:=1-\sigma_{\min}^2(H^{\frac{1}{2}}A), \quad
				\hat{\rho}:=\max\{\hat{\rho}_z, \hat{\rho}_x\},
			\end{equation}
			and
			\begin{equation}\label{gamma}
				\gamma:=\frac{\lambda_{\max}(H)}{\Lambda_{\min}}+\frac{1}{\sigma_{\min}^2(A)}. 
			\end{equation}
			We have the following convergence results for Algorithm \ref{ASEHBM}.
			\begin{theorem}\label{the3}
				Suppose the probability spaces $(\overline{\Omega}, \overline{\mathcal{F}}, \overline{P})$ and $(\Omega, \mathcal{F}, P)$ satisfy Assumption \ref{Ass}. 
				For any given linear system $Ax=b$, let  $\{x^k\}_{k\geq0}$ be the iteration sequence generated by Algorithm \ref{ASEHBM}.
				Then we have
				$$\begin{aligned}
					\mathbb{E}\left[\|x^k-A^{\dagger}b\|_2^2\right]
					\leq
					\hat{\rho}^{k} \left(\|x^0-A^{\dagger}b\|_2^2+\frac{\gamma}{\max\left\{\left|1-\frac{\hat{\rho}_{x}}{\hat{\rho}_{z}}\right|, \frac{1}{k}\right\}} \|z^0-b_{\text{Range}(A)^{\bot}}\|_2^2\right),
				\end{aligned}$$
				where $\hat{\rho}_{z}$,  $\hat{\rho}_{x}$,  $\hat{\rho}$, and $\gamma$ are given by \eqref{rho_z} and \eqref{gamma}, respectively.
			\end{theorem}

			\begin{remark}
				We compare the convergence factor $\hat{\rho}$ in Theorem \ref{the3} for Algorithm \ref{ASEHBM} with the convergence factor $\rho$ in Theorem \ref{x-ASEM-convergence} for Algorithm \ref{ASE}. For the case where the relaxation parameters $\eta=\zeta=1$ in Algorithm \ref{ASE}, we have $\rho_{z}=\hat{\rho}_{z}$ and $\rho_{x}= \hat{\rho}_{x}$. This indicates that $\rho= \hat{\rho}$, i.e., 
				Algorithm \ref{ASEHBM} and Algorithm \ref{ASE} share the same convergence factor.
				%the convergence factor in Theorem \ref{the3} for Algorithm \ref{ASEHBM} coincides with  that of Algorithm \ref{ASE} in Corollary \ref{corollary}. 
				To the best of our knowledge, this is the first time that adaptive heavy ball momentum is integrated into stochastic extended iterative methods for solving general linear systems \eqref{LS}.
			\end{remark}

			\section{Numerical experiments}
			\label{Section-5}
			
			In this section, we implement  Algorithm \ref{ASE} and  Algorithm  \ref{ASEHBM}. We compare our algorithms with the REABK method \eqref{REABK} with a constant step-size \cite{Du20Ran}, as well as with 
			stochastic variance reduced gradient (SVRG) \cite{johnson2013accelerating}, and the {\sc Matlab} functions \texttt{pinv} and \texttt{lsqminnorm}.	All the methods are implemented in  {\sc Matlab} R2022a for Windows $11$ on a desktop PC with Intel(R) Core(TM) i7-1360P CPU @ 2.20GHz  and 32 GB memory. The code to reproduce our results can be found at  \href{https://github.com/xiejx-math/ASEIM-codes}{https://github.com/xiejx-math/ASEIM-codes}.

			\subsection{Inconsistency and variance reduction}
			
			It is well-known that SGD \eqref{SGD-iteration} can exhibit issues such as slow convergence or even failure to converge due to the non-vanishing variance in its gradient estimates.  Specifically, the variance satisfies  
			\[
			\lim_{k \rightarrow \infty} \mathbb{E}\left[\|\nabla f_{i_k} (x^k) - \nabla f(x^k)\|^2_2\right] \neq 0,  
			\]  
			preventing  variance reduction. For instance, consider using SGD to solve the least-squares problem
			\begin{equation}\label{quar-f}  
				\min_{x \in \mathbb{R}^n} f(x) := \frac{1}{2m} \|Ax - b\|_2^2 = \frac{1}{m} \sum_{i=1}^m f_i(x),
			\end{equation}
			where each component function is \( f_i(x) = \frac{1}{2}(A_{i,:}x - b_i)^2 \). For convenience, we assume that $A  $ is normalized to $\|A_{i,:}\|^2_2=1$ for each row of $A$. As shown in Section \ref{section-122}, SGD in this context reduces to the RK method.
			Let \( x^* \) be an optimal solution of \eqref{quar-f}. The variance of the gradient estimate at \( x^* \) is  
			\[
			\sigma^2 = \mathbb{E}\left[\| \nabla f_i(x^*) - \nabla f(x^*)\|^2_2\right] = \mathbb{E}\left[\| \nabla f_i(x^*)\|^2_2\right] = \sum_{i=1}^m \frac{(A_{i,:}x^* - b_i)^2}{m} = \frac{\|r^*\|_2^2}{m},  
			\]  
			where \( r^* = Ax^* - b \) denotes the residual vector. When the system is consistent, as the iterates approach $x^*$, the residual gradually drops to zero and thus so does the variance. This vanishing variance enables SGD to converge even with constant stepsizes. 
			When the system is inconsistent, $r^* \neq 0$ and the variance \( \sigma^2 \) persists, necessitating variance reduction techniques.  In our test, we focus on  SVRG to address this limitation. The detailed iteration scheme of  SVRG for solving least-squares problems are formally described in Algorithm \ref{SVRG}.
			
			\begin{algorithm}[htpb]
				\caption{SVRG for solving the least squares \eqref{LS}. }
				\label{SVRG}
				\begin{algorithmic}
					\Require
					$A\in \mathbb{R}^{m\times n}$, $b\in \mathbb{R}^m$, the positive integer $N$, $k=0$, the initial point $x^{0}\in\operatorname{Range}(A^\top)$, and the  step-size $\alpha$.
					\begin{enumerate}
						\item[1:] Set $\tilde{x}=x^{k}$, $\tilde{u}=\frac{1}{m}A^{\top}(A\tilde{x}-b)$, and $x^{0}=\tilde{x}$.
						\item[2:] For $t=0, 1, \cdots,N-1$
						
						\qquad Randomly select an index $i_{t} \in [m]$ with probability $\frac{1}{m}$.
						
						\qquad Compute $g^{t}=(A_{i_{t},:}(x^{t}-{x}^{0}))A_{i_{t},:}^{\top}+\tilde{u}$.
						
						\qquad Update $x^{t+1}=x^{t}-\alpha g^{t}$.
						
						End for
						\item [3:] \textbf{Option I:} \;$x^{k}=x^{N}$.
						\item [4:] \textbf{Option II:} \;$x^{k}=x^{t}$ for randomly chosen $t \in \{0, 1, \cdots, N-1\}$.
						\item[5:] If the stopping rule is satisfied, stop and go to output. Otherwise, set $k=k+1$ and go to Step $1$.
					\end{enumerate}
					
					\Ensure
					The approximate solution $x^k$.
				\end{algorithmic}
			\end{algorithm}

			We emphasize that the REK-type methods inherently possess variance reduction properties. Specifically, let us consider	Algorithm \ref{ASE}.  At the $k$-th iteration, its update for  $x^{k+1}$ can be interpreted as performing a single SGD step for the stochastic optimization problem
			$$
			\mathop{\operatorname{\min}}\limits_{x \in \mathbb{R}^{n}} \mathop{\mathbb{E}}\limits_{S\in\Omega}[f_{S}(x, z^{k+1})]=\frac{1}{2}\mathop{\mathbb{E}}\limits_{S\in\Omega}[\|S^{\top}(Ax-(b-z^{k+1}))\|_{2}^{2}].
			$$
			Since $z^{k+1}$ converges to $b_{\text{Range}(A)^{\perp}}=AA^\dagger b-b$  independently of the iteration sequence $\{x^{k}\}_{k \geq 0}$ (see Lemma \ref{z-the basic method-convergence}), it follows that
			$$
			\mathbb{E}\left[\| 	\nabla f_{S}(x^{*}, z^{k+1}) -\mathbb{E}[\nabla f_{S}(x^{*}, z^{k+1})] \|^2_2\right]=\mathbb{E}\left[\left\| A^\top (SS^\top-\mathbb{E}[SS^\top])(AA^\dagger b-b-z^{k+1})\right\|^2_2\right]
			$$
			vanishes asymptotically. This demonstrates that  Algorithm \ref{ASE} achieves variance reduction.

			\subsection{Numerical setup}

			We consider two types of matrices. One is the random Gaussian matrices generated by the {\sc Matlab} function {\tt randn}. Specifically, for given $m, n, r$, and $\kappa>1$, we construct a dense matrix $A$ by $A=U D V^\top$, where $U \in \mathbb{R}^{m \times r}, D \in \mathbb{R}^{r \times r}$, and $V \in \mathbb{R}^{n \times r}$. Using {\sc Matlab}  notation, these matrices are generated by {\tt [U,$\sim$]=qr(randn(m,r),0)}, {\tt [V,$\sim$]=qr(randn(n,r),0)}, and {\tt D=diag(1+($\kappa$-1).*rand(r,1))}. So the condition number and the rank of $A$ are upper bounded by $\kappa$ and $r$, respectively. The other is  real-world data which are available via SuiteSparse Matrix Collection \cite{Kol19} and LIBSVM \cite{chang2011libsvm}, where only the coefficient matrices $A$ are employed in the experiments.

			In our numerical experiments, we focus exclusively on inconsistent linear systems. For consistent systems, RK-type methods naturally outperform REK-type methods due to the additional computational cost introduced by the auxiliary sequence \(\{z^k\}_{k \geq 0}\) in REK-type methods. While RK-type methods are highly effective for solving consistent systems, REK-type methods are specifically designed to address their limitations in handling inconsistent systems. To construct an inconsistent linear system, we set $b=A x+b_e$, where $x$ is a vector with entries generated from a standard normal distribution and $b_e \in \operatorname{Null}\left(A^\top\right)$. 
			The computations are terminated once the relative solution error (RSE), defined as
			$\text{RSE}=\|x^k-A^\dagger b\|^2_2/\|x^0-A^\dagger b\|^2_2$, is less than a specific error tolerance or the number of iterations exceeds a certain limit. 
			In practice, we consider a variable as zero when it is less than  \texttt{eps}.
			For each experiment, we run $20$ independent trials.

			As for the probability spaces, we consider the following partitions of $[m]$ and $[n]$ which are commonly used in the context \cite{Nec19,necoara2022stochastic,Du20Ran,zeng2023adaptive}, $[m] = \left\{\mathcal{I}_1, \mathcal{I}_2, \dots, \mathcal{I}_\ell\right\}$, $[n]=\left\{\mathcal{J}_1, \mathcal{J}_2, \dots, \mathcal{J}_t\right\}$, where
			$$	
			\begin{aligned}
				\mathcal{I}_i&=\left\{\varpi_m(k): k=(i-1)p+1,(i-1)p+2,\ldots,ip\right\}, i=1, 2, \ldots, \ell-1,
				\\
				\mathcal{I}_t&=\left\{\varpi_m(k): k=(\ell-1)p+1,(t-1)p+2,\ldots,m\right\}, |\mathcal{I}_\ell|\leq p,\\
				\mathcal{J}_j&=\left\{\varpi_n(k): k=(j-1)p+1,(j-1)p+2,\ldots,jp\right\}, j=1, 2, \ldots, t-1,
				\\
				\mathcal{J}_t&=\left\{\varpi_n(k): k=(t-1)p+1,(t-1)p+2,\ldots,n\right\}, |\mathcal{J}_t|\leq p.
			\end{aligned}$$
			Here $\varpi_m$ and $\varpi_n$ are uniform random permutations on $[m]$ and $[n]$, respectively, and $p$ is the block size. The sampling matrices are chosen according to the rules discussed in Remark \ref{remark_sampling}.
			With this setup, when $\eta=\zeta=1$, Algorithms
			\ref{ASE} and \ref{ASEHBM} lead to the REABK with adaptive step-sizes (AREABK) and REABK with adptive HBM (AmREABK), respectively. Table \ref{table-p} provides the parameter configurations for REABK, SVRG, AREABK, and AmREABK, which are consistent with the established settings in prior works \cite{Du20Ran,roux2012stochastic, johnson2013accelerating, nguyen2017sarah}.

			\begin{table}
				\renewcommand\arraystretch{1.5}
				\setlength{\tabcolsep}{14pt}
				\caption{Parameter configurations for REABK, SVRG, AREABK, and AmREABK, where $\Gamma_{\max}:=\max\{\Gamma_{\max}^{\mathcal{I}},\Gamma_{\max}^{\mathcal{J}}\}$ with $\Gamma_{\max}^{\mathcal{I}}$ and $\Gamma_{\max}^{\mathcal{J}}$ being given in Remark \ref{remark_sampling}, and $L:=\mathop{\max}\limits_{ i \in[m] } \{\|A_{i, :}\|_{2}^{2}\}$.}
				\label{table-p}
				\centering
				{\scriptsize
					\begin{tabular}{  |c| l|   }
						\hline
						%{ Name }& { Parameter $\mu_{k}$ } & { Parameter $\alpha_{k}$ } &  { Parameter $\omega_{k}$ } & { Parameter $\beta_{k}$ } & {Reference}
						{ Solvers  }& { Parameters   }  
						\\
						\hline
						\multirow{2}{*}
						{REABK \cite{Du20Ran}}& $\alpha=1.75/\Gamma_{\max} $ (randomized Gaussian matrices) \\ 
						&     $\alpha=1/\Gamma_{\max}$ (real-world data)   \\
						%\hline
						%{SAGA \cite{defazio2014saga} }& $\alpha=0.5/L$    \\
						\hline
						{SVRG \cite{johnson2013accelerating}}& $\alpha=0.1/L, N=2m$   \\
						\hline
						AREABK(Algorithm \ref{ASE}) & $\eta=1, \zeta=1$   \\
						\hline
						{AmREABK} (Algorithm \ref{ASEHBM}) & Stages  I and   II   \\
						\hline
					\end{tabular}
				}
			\end{table}

			\subsection{Impact of block size $p$} 
			In this subsection, we  investigate the impact of the block size $p$ on the convergence of REABK, AREABK, and AmREABK with random Gaussian matrices.
			We present three different scenarios: well-conditioned, ill-conditioned, and rank-deficient coefficient matrices. % for the four methods.
			The performance of the algorithms is measured in the computing time (CPU), the number of full iterations $(k\cdot\frac{p}{m})$, and the actual convergence factor. The number of full iterations makes the number of operations for one pass through the rows of $A$ are the same for all the algorithms. And the actual convergence factor is defined as 
			$$\rho_{\text{actual}} = \left(\frac{\|x^K-A^\dagger b\|^2_2}{\|x^0-A^\dagger b\|^2_2}\right)^{1/K},$$ 
			where $K$ is the number of iterations when the algorithm terminates. 
			
			The results are displayed in Figure \ref{figureR1}.
			The bold line illustrates the median value over $50$ trials. The lightly shaded area signifies the  range from the minimum to the maximum values, while the darker shaded one indicates the data lying between the $25$-th and $75$-th quantiles. Additionally, we have incorporated a plot of the upper bound $\rho_{3}$, as discussed in Remark \ref{remark_sampling}, for a comprehensive comparison. It  can be observed that both  REABK, AREABK, and AmREABK display smaller convergence factors  compared to the theoretical bound $\rho_{3}$.
			
			Figure \ref{figureR1} illustrates that AmREABK requires fewer full iterations and less CPU time as the block size $p$ increases.
			In contrast, both REABK and AREABK require a greater number of full iterations. This can be attributed to the fact that the actual convergence factor of AmREABK decreases more significantly than those of REABK and AREABK as $p$ increases. In fact, it can be observed that AmREABK  consistently outperforms  AREABK  in terms of the number of full iterations, regardless of the condition number of $A$ and whether it is full rank. However, when the block size is relatively small (e.g. $p\leq 32$), AREABK demonstrates superior performance to AmREABK in terms of CPU time. This is because AmREABK demands more computation costs at each step compared to AREABK. Moreover, when the block size is relatively large (e.g. $p\geq16$), both AREABK and  AmREABK outperform REABK in terms of both the number of full iterations and CPU time.
			
			\begin{figure}[hptb]
				\centering
				\begin{tabular}{cc}
					\includegraphics[width=0.33\linewidth]{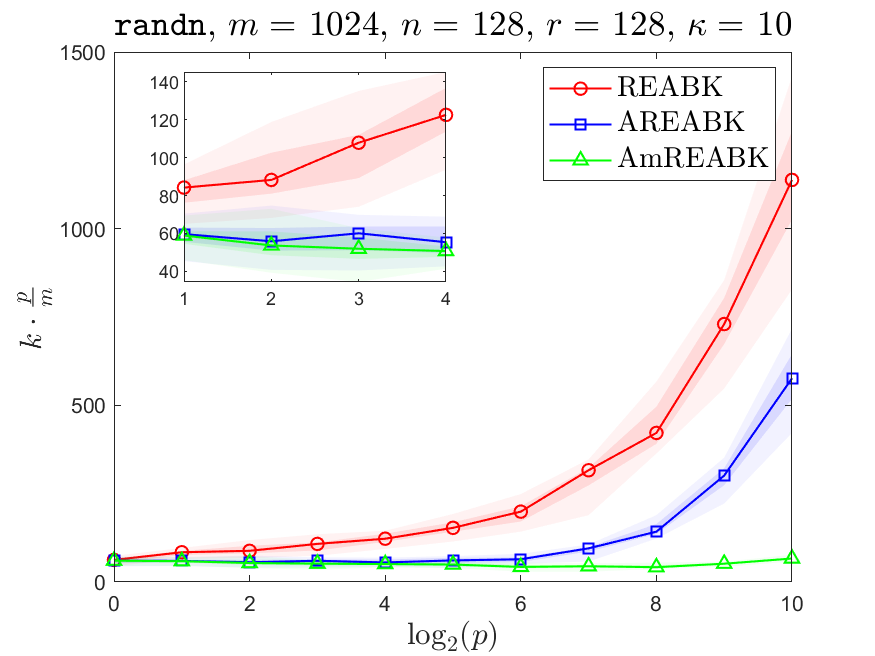}
					\includegraphics[width=0.33\linewidth]{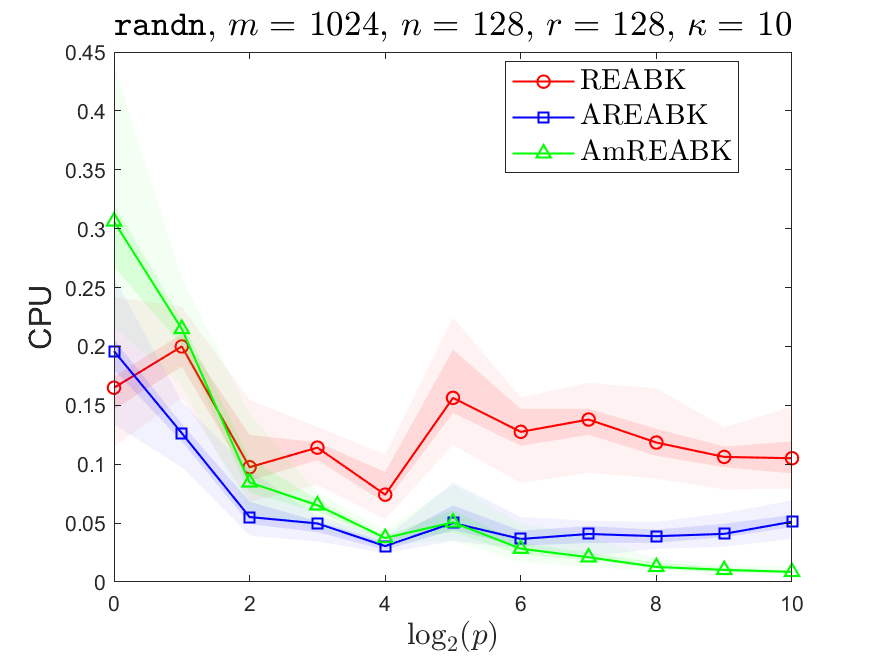}
					\includegraphics[width=0.33\linewidth]{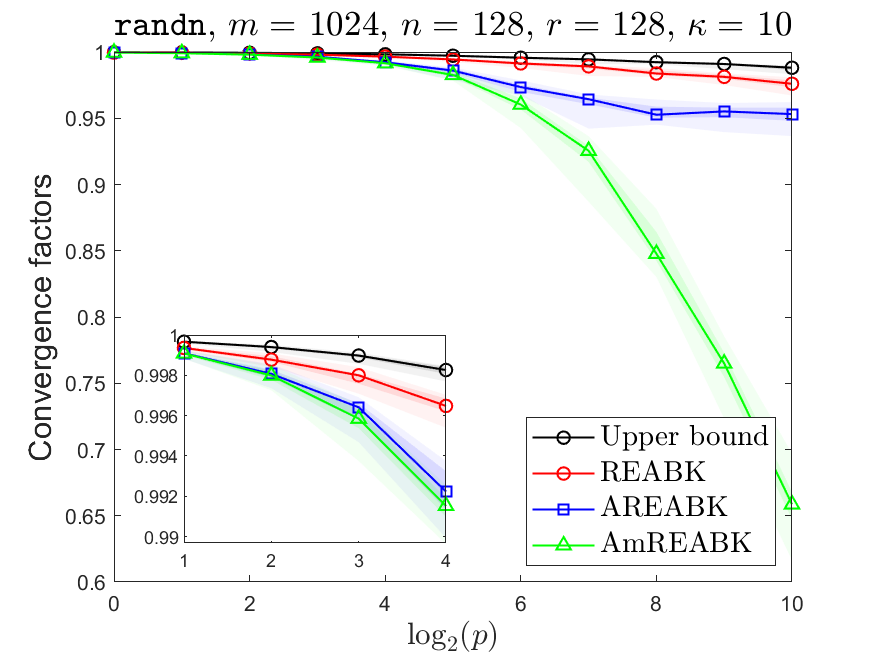}\\
					\includegraphics[width=0.33\linewidth]{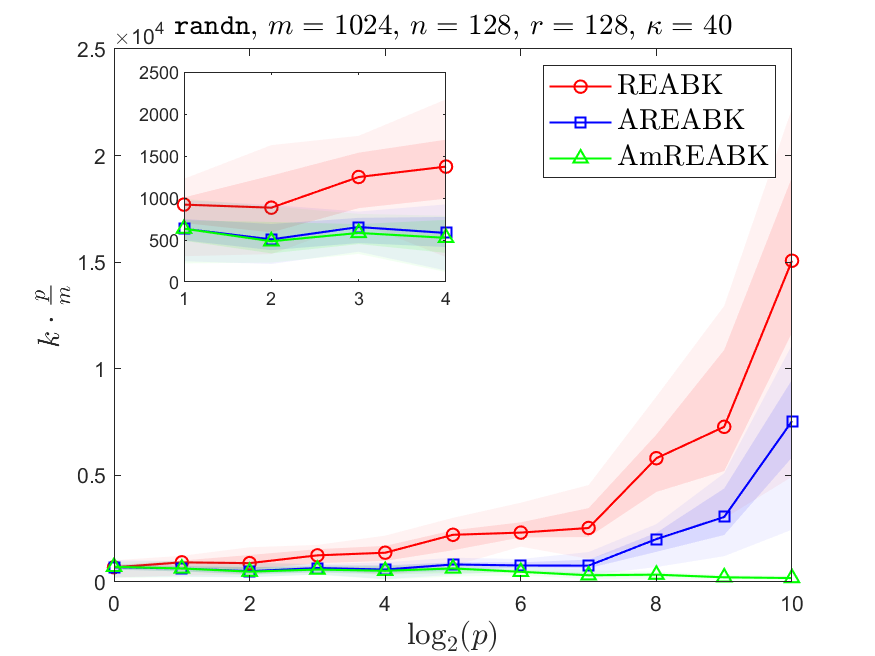}
					\includegraphics[width=0.33\linewidth]{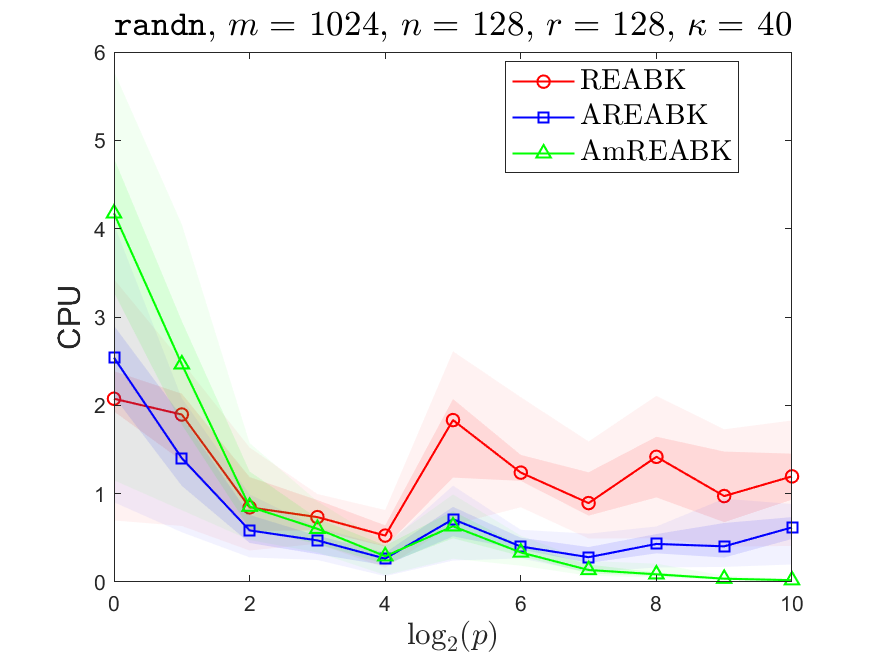}
					\includegraphics[width=0.33\linewidth]{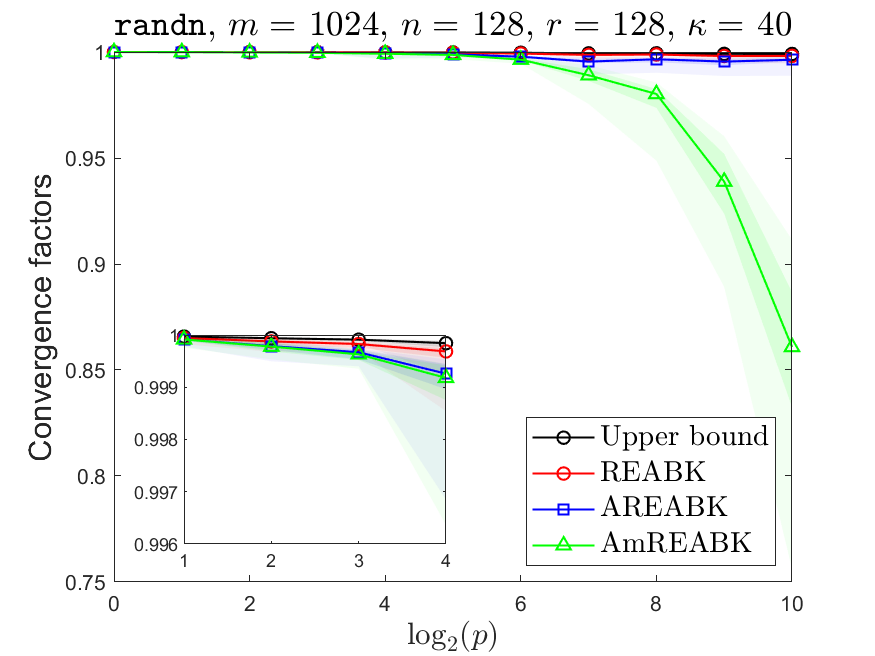}\\
					\includegraphics[width=0.33\linewidth]{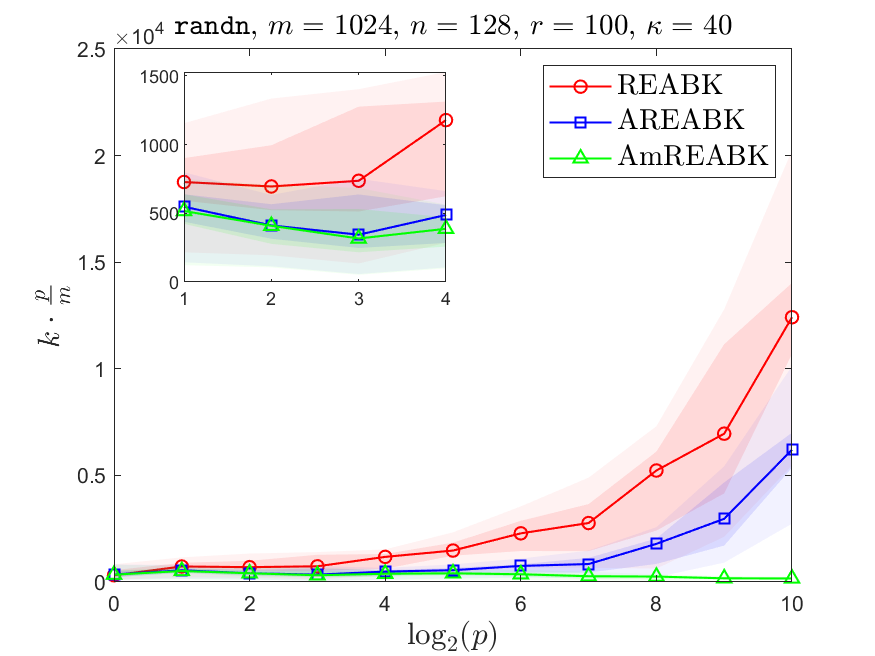}
					\includegraphics[width=0.33\linewidth]{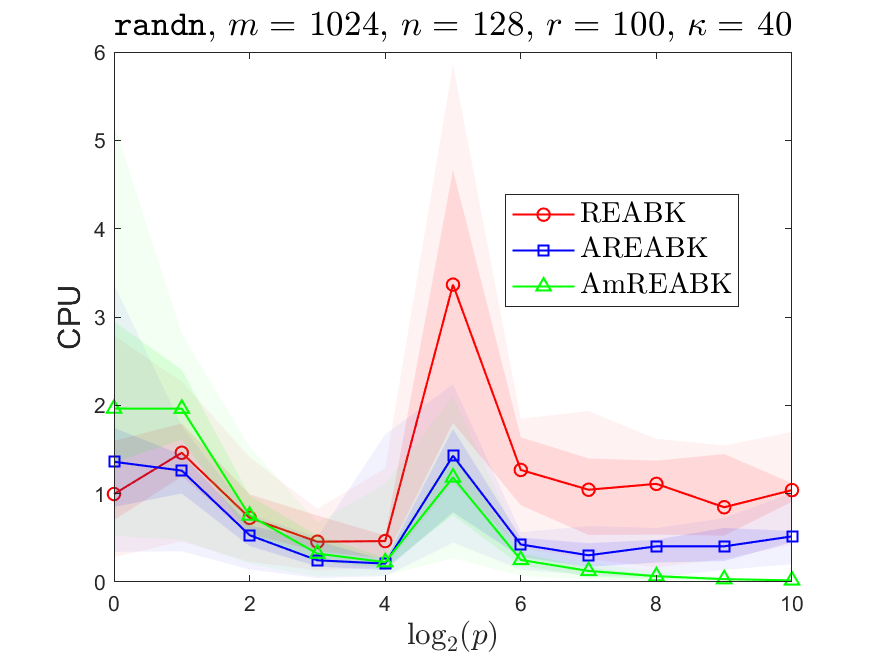}
					\includegraphics[width=0.33\linewidth]{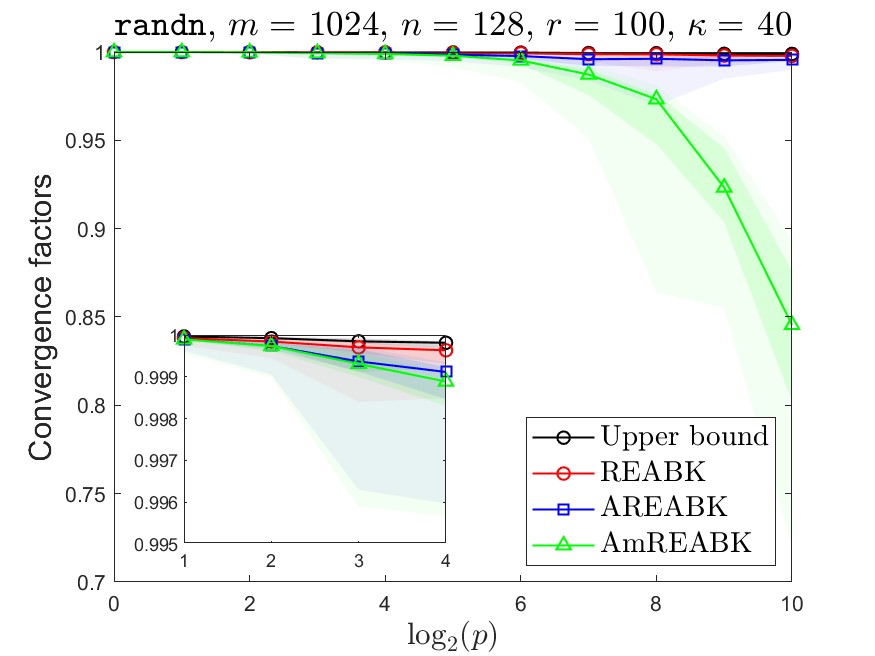}
				\end{tabular}
				\caption{
					Figures illustrate the evolution of the number of full iterations (left),  the CPU time (middle), and the actual convergence factor (right) with respect to the block size $p$.  The title of each plot indicates the values of $m,n,r$, and $\kappa$. All computations are terminated once $\text{RSE}<10^{-12}$. }
				\label{figureR1}
			\end{figure}

			\subsection{Comparison to REABK and SVRG}

			This subsection presents a performance comparison among  SVRG, REABK, AREABK, and AmREABK. We assess these methods through both iteration numbers and CPU time, focusing on their performance with different types of coefficient matrices. For SVRG, we employ Option I for updating the next iterate.
			
			In Figure \ref{Rfigure2}, we present the CPU time comparisons for  SVRG, REABK, AREABK, and AmREABK with random Gaussian matrices. The results reveal that the REK-type methods (REABK, AREABK, and AmREABK) consistently outperform SVRG. Notably, SVRG demonstrates greater sensitivity to the condition number of the coefficient matrix.
			
			It can be seen that the CPU time for REK-type methods exhibits minimal variation as the matrix size increases. This stability highlights their scalability advantage: the number of iterations required remains nearly constant, ensuring consistent efficiency regardless of problem dimension. Figure \ref{RfigureM2} further validates this by showing that the iteration count for REK-type methods remains effectively unchanged as the number of matrix rows increases. From Figures \ref{Rfigure2} and \ref{RfigureM2}, it is clear that among REABK, AREABK, and AmREABK, the AmREABK method is the most competitive.

			\begin{figure}[hptb]
				\centering
				\begin{tabular}{cc}
					\includegraphics[width=0.33\linewidth]{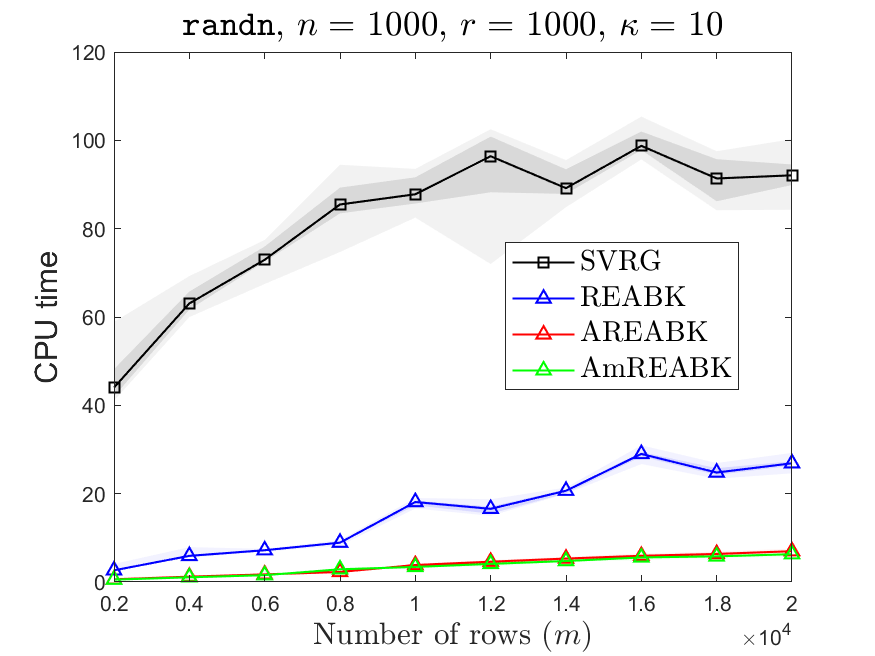}
					\includegraphics[width=0.33\linewidth]{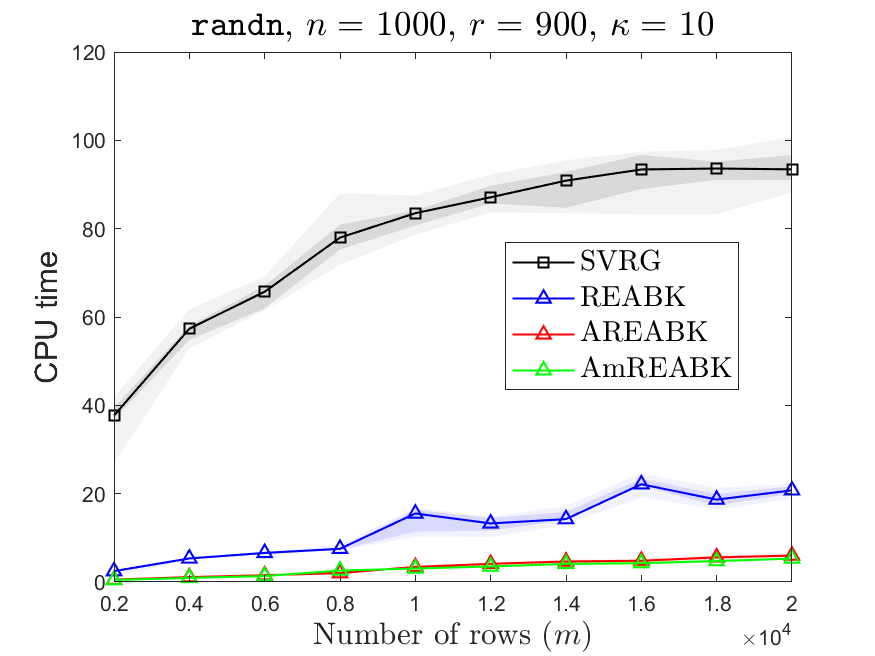}
					\includegraphics[width=0.33\linewidth]{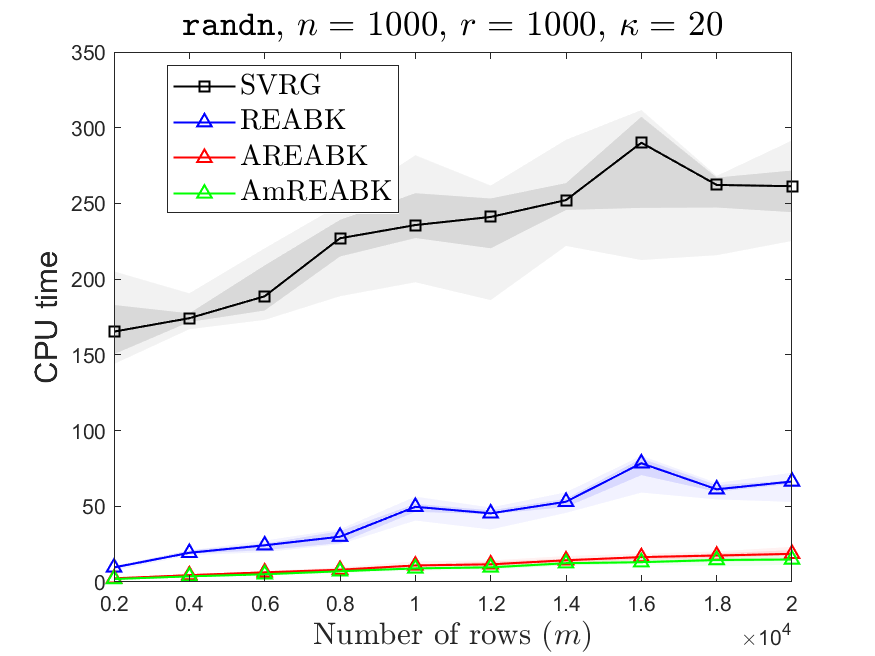}
				\end{tabular}
				\caption{
					Figures illustrate the evolution of  CPU time (in seconds) vs increasing number of rows.  The title of each plot indicates the values of $m,n,r$, and $\kappa$. We set $p=300$ for REABK, AREABK, and AmREABK.  All computations are terminated once $\text{RSE}<10^{-12}$.}
				\label{Rfigure2}
			\end{figure}
			
			\begin{figure}[tbhp]
				\centering
				\begin{tabular}{cc}
					\includegraphics[width=0.33\linewidth]{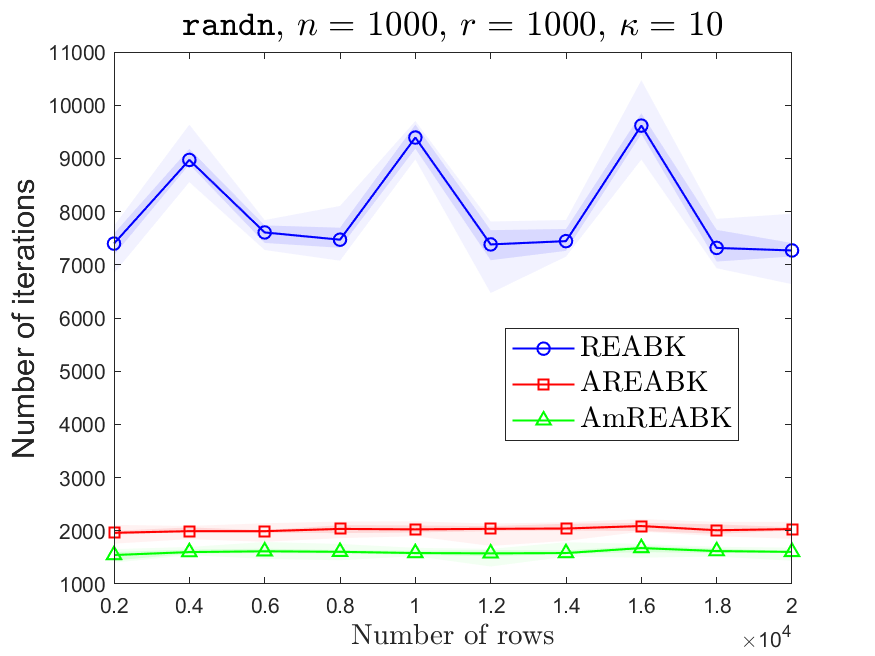}
					\includegraphics[width=0.33\linewidth]{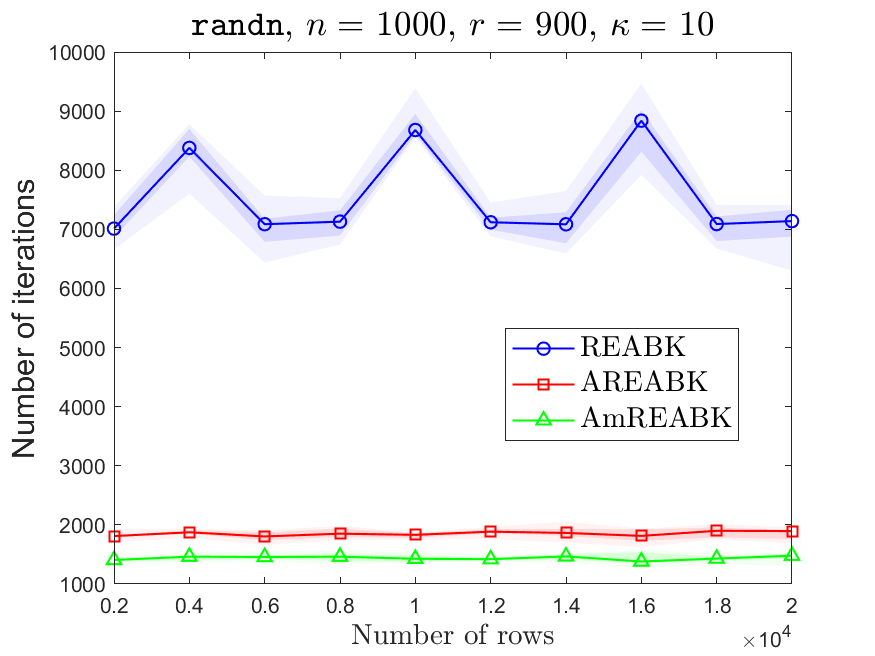}
					\includegraphics[width=0.33\linewidth]{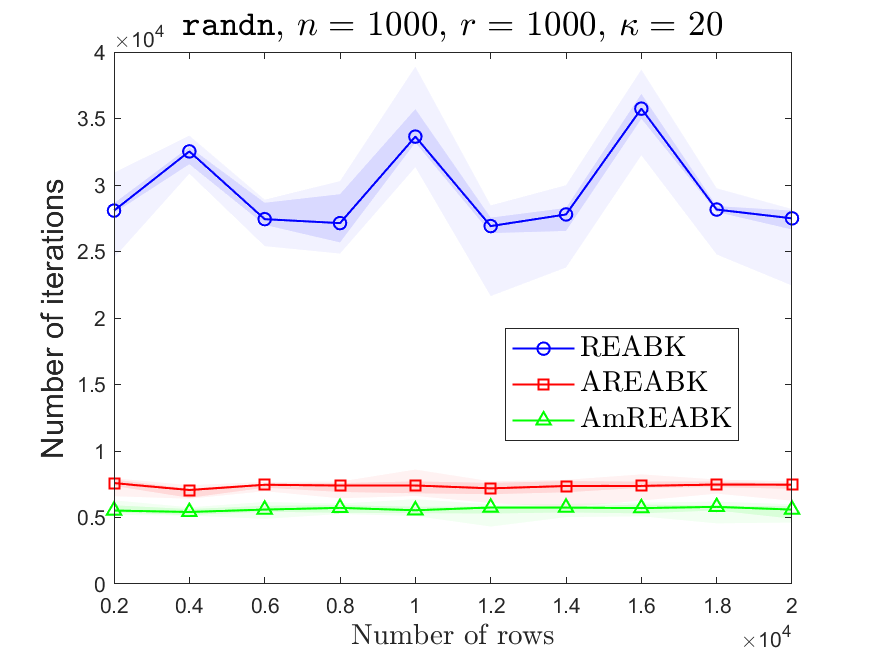}
				\end{tabular}
				\caption{Figures illustrate the number of iterations vs increasing number of rows. We set $p=300$ and stop the
					algorithms once $\text{RSE}<10^{-12}$. The title of each
					plot indicates the values of $n,r$ and $\kappa$.  }
				\label{RfigureM2}
			\end{figure}
			
			We note that SVRG has practical limitations due to its need for periodic full matrix-vector products.  These requirements are impractical for large-scale linear systems where \(A\) cannot be stored in RAM. In fact, it has been observed that the randomized extended-type methods consistently outperform   SVRG in terms of CPU time throughout our experiments. Consequently, we will focus our subsequent tests on the REK-type methods.
			
			Table \ref{table1} and Figure \ref{figureR2} present the number of  iterations and the computing time for REABK, AREABK, and AmREABK when applied to sparse matrices from SuiteSparse Matrix Collection and LIBSVM. The six matrices from SuiteSparse Matrix Collection are {\tt bibd\_16\_8}, {\tt crew1}, {\tt WorldCities},  {\tt model1}, {\tt ash958}, and {\tt Franz1}, and the three from LIBSVM are {\tt a9a}, {\tt cod-rna}, and {\tt ijcnn1}. This selection covers a range of matrix properties, including full-rank and rank-deficient, well-conditioned and ill-conditioned,  overdetermined and underdetermined.
			
			Table \ref{table1} shows the numerical results on the matrices from SuiteSparse Matrix Collection.
			As displayed in Table \ref{table1}, AmREABK is consistently better than AREABK in terms of iterations across all test cases. However, we note that AREABK  may outperform AmREABK in terms of CPU time in specific cases, such as {\tt ash958} and {\tt Franz1}. This is because  AmREABK requires more computations at each step compared to REABK.
			Moreover, Table \ref{table1} also reveals that both AmREABK and AREABK outperform REABK in terms of both the number of iterations and CPU time for all test cases. 
			Figure \ref{figureR2} presents the numerical results on the matrices from LIBSVM. 
			Our results show that AmREABK still outperforms REABK and AREABK. Notably,  for dataset {\tt cod-rna}, it can be observed that only AmREABK successfully finds the solution, while both REABK and AREABK fail.

			\begin{table}
				\renewcommand\arraystretch{1.5}
				\setlength{\tabcolsep}{2pt}
				\caption{ The average Iter and CPU of REABK, AREABK, and AmREABK for linear systems with coefficient matrices from SuiteSparse Matrix Collection \cite{Kol19}. 
					We set the block size $p=30$ and stop the
					algorithms once $\text{RSE}<10^{-12}$.}
				\label{table1}
				\centering
				{\scriptsize
					\begin{tabular}{  |c| c| c| c| c |c |c |c | c|c|  }
						\hline
						\multirow{2}{*}{ Matrix}& \multirow{2}{*}{ $m\times n$ }  &\multirow{2}{*}{rank}& \multirow{2}{*}{$\frac{\sigma_{\max}(A)}{\sigma_{\min}(A)}$}  &\multicolumn{2}{c| }{REABK}  &\multicolumn{2}{c| }{AREABK} &\multicolumn{2}{c| }{AmREABK}
						\\
						\cline{5-10}
						& &   &    & Iter & CPU    & Iter & CPU   & Iter & CPU     \\%[0.1cm]
						\hline
						{\tt bibd\_16\_8}& $120\times12870$ &  120  & 9.54 &     4082.62 &   2.3 &  2809.80 &   1.72 &  2150.14 &   {\bf 1.61 }   \\
						\hline
						{\tt crew1} & $135\times6469 $ &  135  &18.20  & 30092.90 &   5.51 &  3844.94 &   0.48 &  3380.62 &   {\bf 0.45} \\
						\hline
						{\tt WorldCities} & $315\times100$ &  100  &6.60  & 70816.16 &   0.43 & 12551.30 &   0.08 &  3426.90 &  {\bf 0.02}  \\
						\hline
						{\tt model1} & $  362\times798 $ &  362  & 17.57 &84087.38 &   0.50 &  8153.02 &   0.051 &  6275.02 &   {\bf 0.049}\\
						\hline
						{\tt ash958} & $958\times292$ &  292  &3.20  &  2931.34 &   0.029 &   991.16 &   {\bf0.0044 } &   957.54 &   0.0054 \\
						\hline
						{\tt Franz1} & $ 2240\times768 $ &  755  & 2.74e+15 & 10040.46 &   0.31 &  3138.16 &  {\bf 0.043} &  3063.08 &   0.068 \\
						\hline
					\end{tabular}
				}
			\end{table}

			\begin{figure}[hptb]
				\centering
				\begin{tabular}{cc}
					\includegraphics[width=0.33\linewidth]{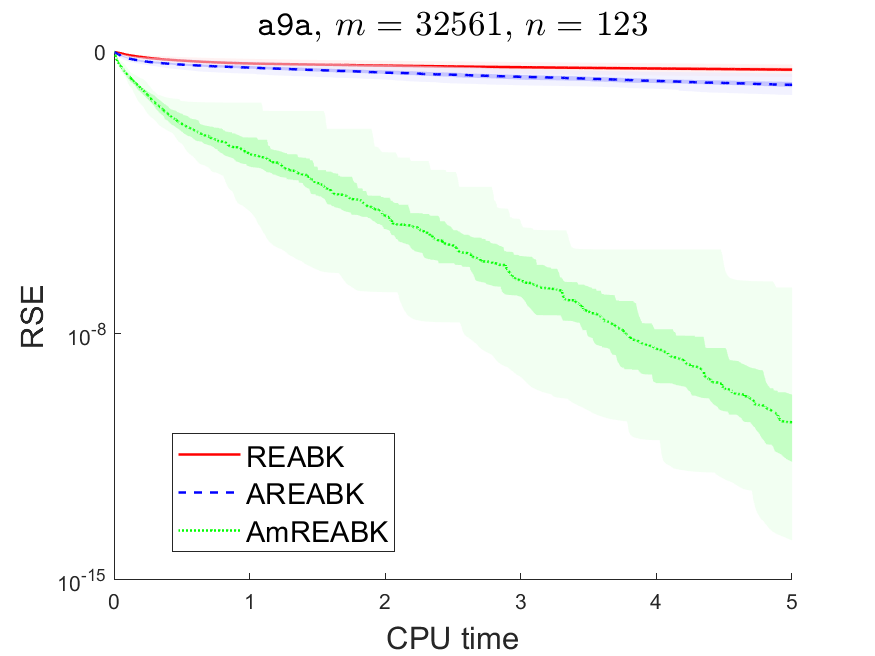}
					\includegraphics[width=0.33\linewidth]{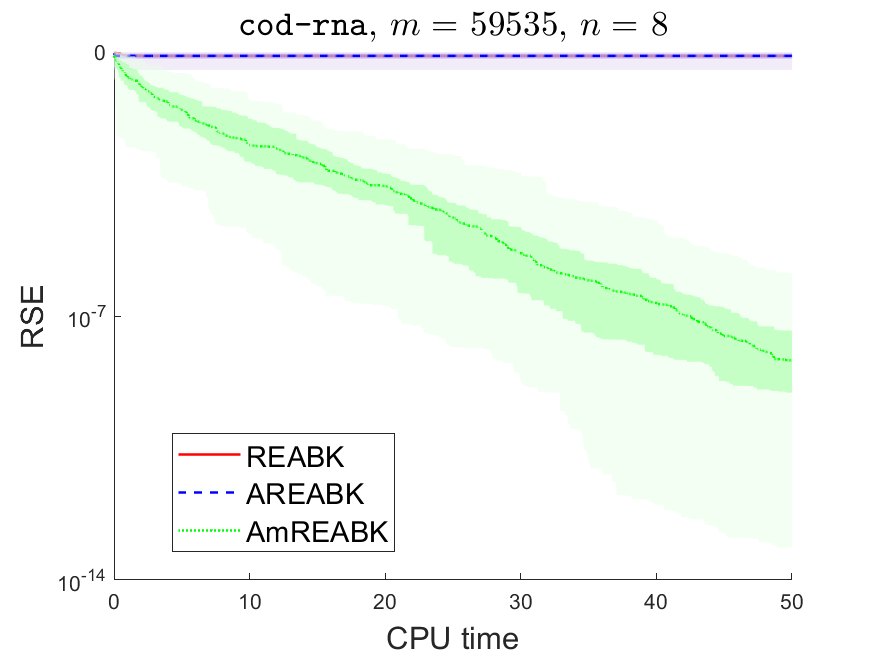}
					\includegraphics[width=0.33\linewidth]{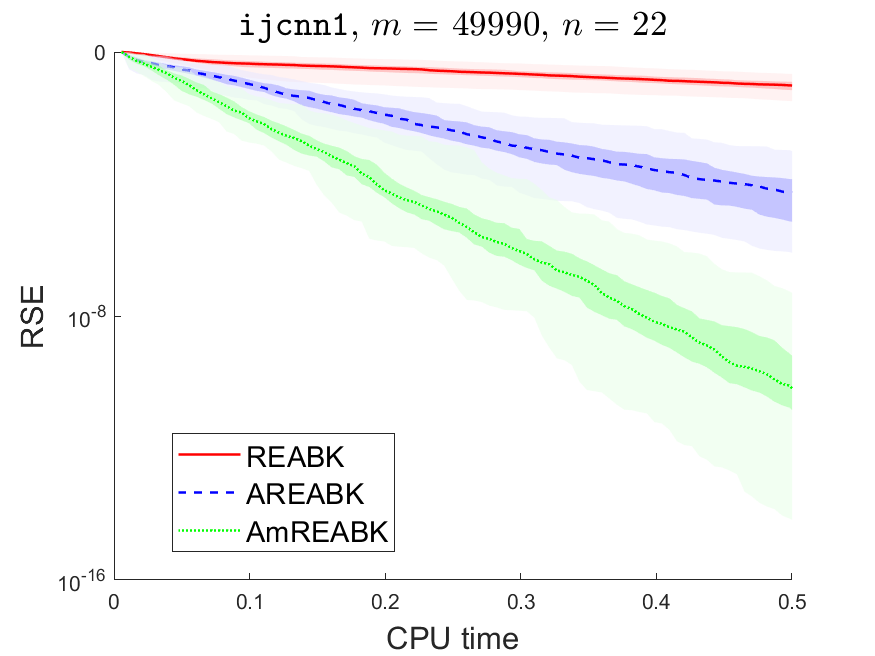}\\
					\includegraphics[width=0.33\linewidth]{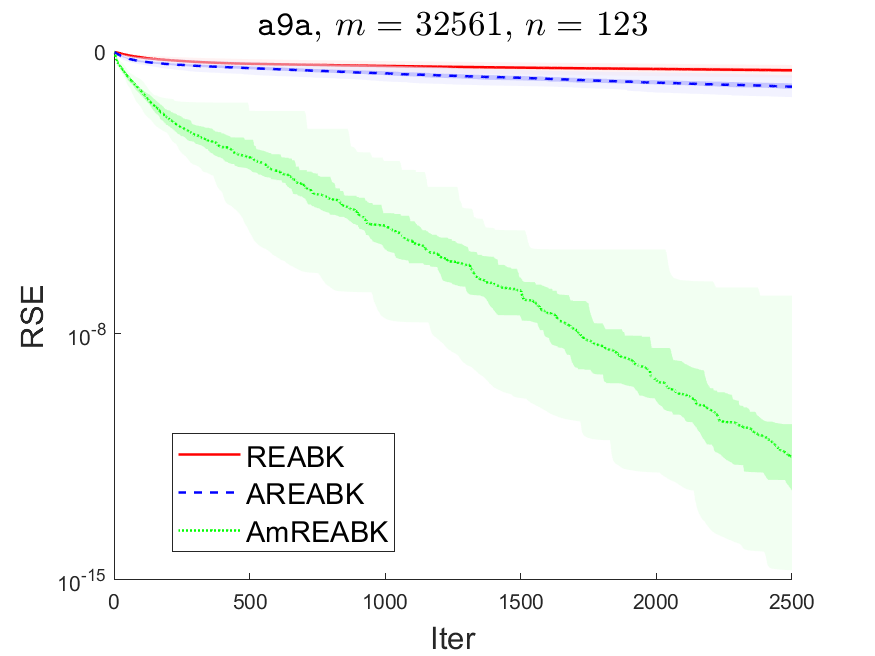}
					\includegraphics[width=0.33\linewidth]{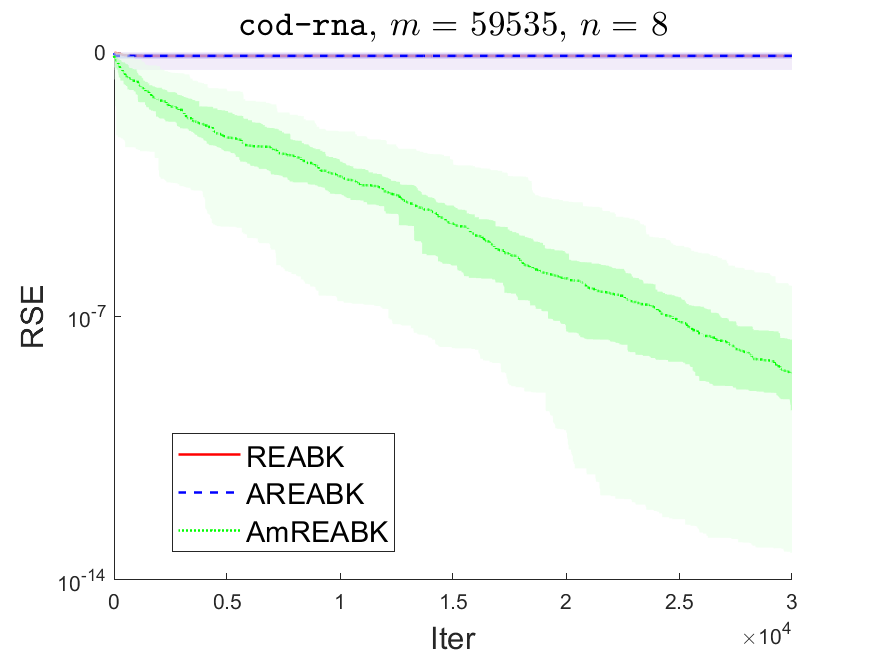}
					\includegraphics[width=0.33\linewidth]{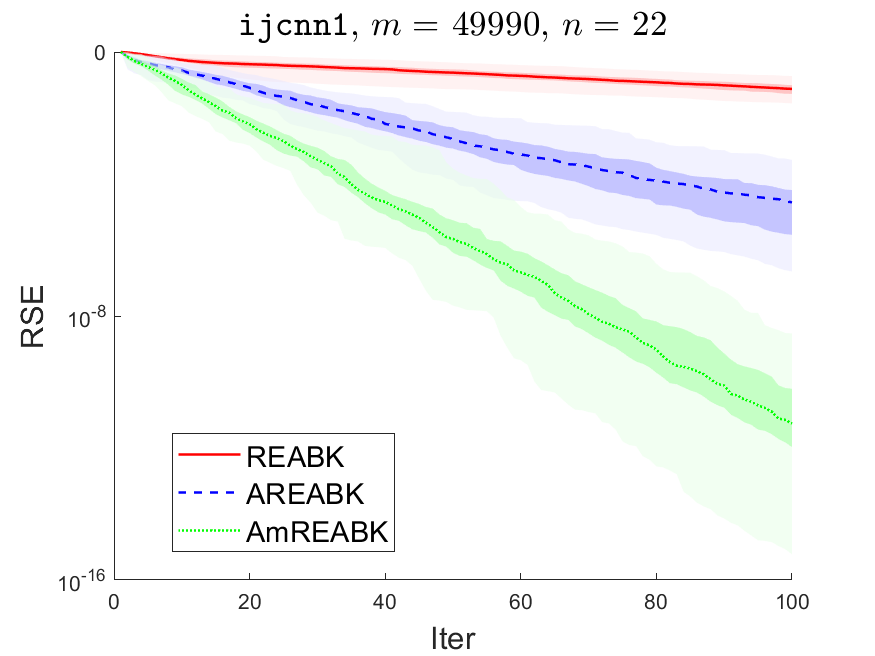}
				\end{tabular}
				\caption{Performance of REABK, AREABK, and AmREABK for linear systems with coefficient matrices from  LIBSVM \cite{chang2011libsvm}.
					Figures illustrate the evolution of RSE with respect to the number of iterations and the CPU time (in seconds).
					The title of each plot indicates the names and sizes of the data.  We set the block size $p=300$ and stop the
					algorithms if the number of iterations exceeds a certain limit.  }
				\label{figureR2}
			\end{figure}

			\subsection{Comparison to {\tt pinv} and {\tt lsqminnorm}}
			
			Next,  we compare the performance of  AmREABK with  {\sc Matlab} functions  {\tt pinv} and {\tt lsqminnorm}.
			To obtain the minimum Euclidean norm least-squares solution easily,  we use the randomly generated matrices with full column rank, where $m\geq n$ and $r=n$. 
			This guarantees that $x^*$ is the desired unique minimum Euclidean norm least-squares solution.
			
			We terminate AmREABK if the accuracy of its approximate solution is comparable to that of the approximate solution obtained by {\tt pinv} and {\tt lsqminnorm}. 
			Figure \ref{RfigureM1} illustrates the CPU time against the number of rows with number of columns fixed. The results show that for smaller coefficient matrices $A$, {\tt pinv} and {\tt lsqminnorm} outperform AmREABK. However, as the scale of $A$ increases, AmREABK exhibits superior computational efficiency.
			This benefit arises from two main reasons. First, as demonstrated in Figure \ref{RfigureM2},  AmREABK exhibits a scalability advantage in that the number of iterations remains nearly constant as the problem dimension grows. Second, in implementation, AmREABK can avoid time-consuming row extraction operations by pre-storing the submatrices of $A$ divided by the partition, thereby further improving the practical performance.

			\begin{figure}[tbhp]
				\centering
				\begin{tabular}{cc}
					\includegraphics[width=0.33\linewidth]{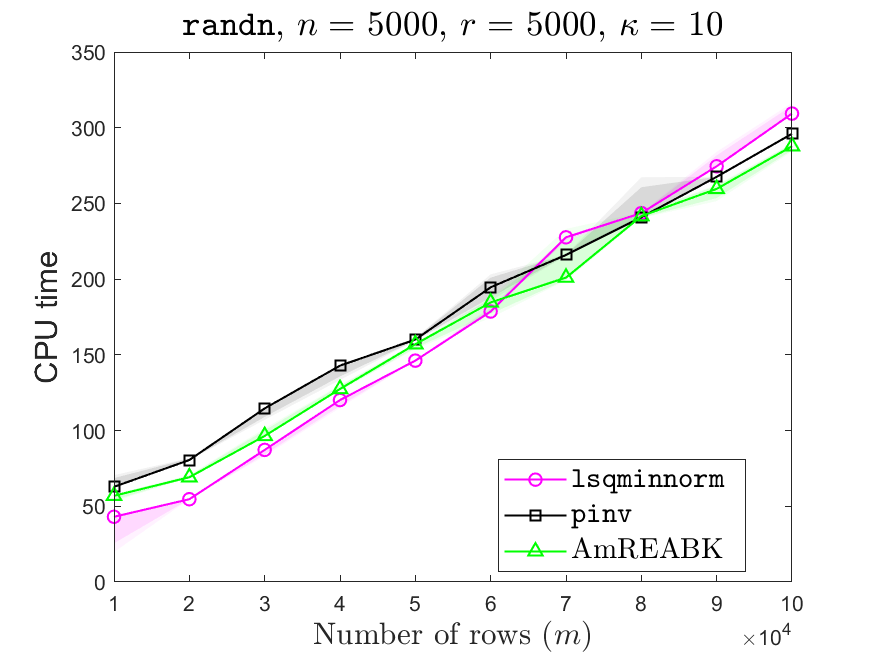}
					\includegraphics[width=0.33\linewidth]{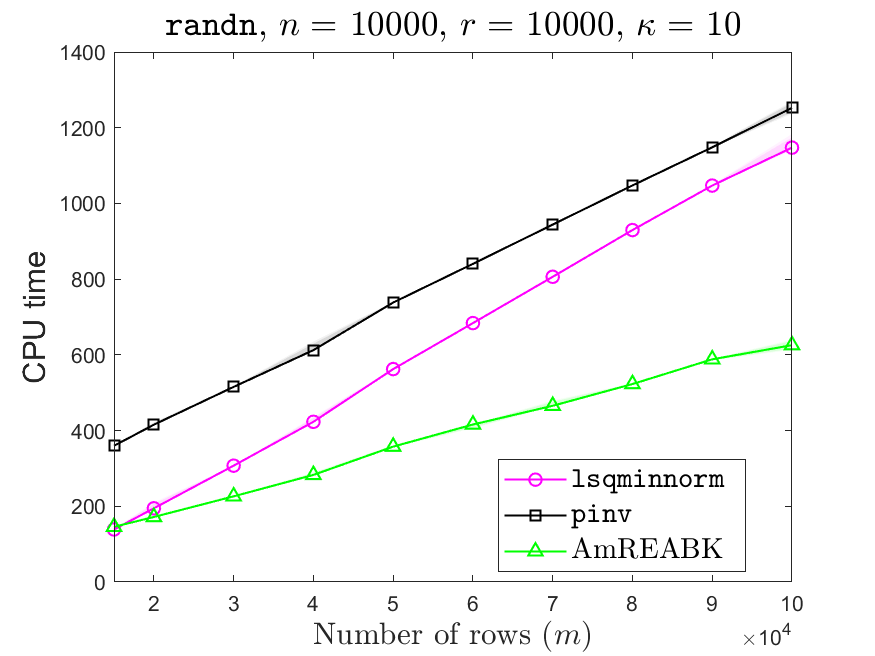}
					\includegraphics[width=0.33\linewidth]{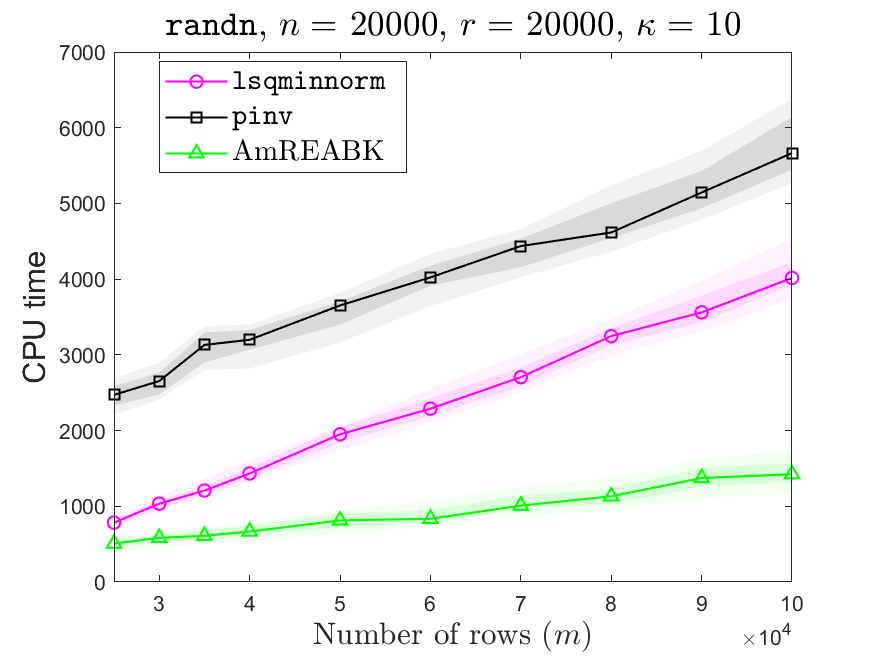}
				\end{tabular}
				\caption{Figures illustrate the CPU time (in seconds) vs increasing number of rows. We set $p=500$.  The title of each
					plot indicates the values of $n,r$ and $\kappa$.  }
				\label{RfigureM1}
			\end{figure}
			
			\section{Concluding remarks}
			\label{Section-6}
			
			We have developed a simple and versatile stochastic extended iterative  algorithmic framework for finding the unique minimum norm least-squares solution for any linear system.
			Notably, we have incorporated the HBM technique into our algorithmic framework. We investigated to adaptively learn the parameters for the momentum variant of this method. Numerical results have confirmed the efficiency of our adaptive stochastic extended iterative method.
			
			The randomized sparse Kaczmarz method proposed in \cite{schopfer2019linear} has been recognized as an effective strategy for obtaining sparse solutions to linear systems. A potential avenue for future research could be the extension of our adaptive stochastic extended iterative method to address sparse recovery problems, especially in situations where linear measurements are corrupted by noise. Additionally, Nesterov's momentum \cite{nesterov1983method,nesterov2003introductory} has gained popularity as a momentum acceleration technique, and recent studies have introduced variants of Nesterov's momentum for accelerating stochastic optimization algorithms \cite{lan2020first}. Research on  the incorporation of Nesterov's momentum into our stochastic extended iterative method would also be a valuable topic.
			
			\bibliographystyle{amsplain}
			\bibliography{zeng2023}

			\section{Appendix. Proof of the main results}
			
			\subsection{Proof of Theorem \ref{x-ASEM-convergence}}
			
			At the $k$-th iteration, we consider the product probability space $\left( \prod_{i=0}^k (\overline{\Omega} \times \Omega), \bigotimes_{i=0}^k (\overline{\mathcal{F}} \otimes \mathcal{F}), \tilde{P} \right)$, where $\times$ denotes the Cartesian product of spaces, $\otimes$ denotes the product of $\sigma$-algebras, and $\tilde{P}$ denotes the corresponding product measure \cite[Section 5]{athreya2006measure}.
			Let $\mathcal{B}_k := (T_0,S_0,\ldots,T_{k-1},S_{k-1})$ and $\overline{\mathcal{B}}_k := (T_0,S_0,\ldots,T_{k-1},S_{k-1},T_k)$ be two random variables in this probability space, where $\mathcal{B}_0$ denotes the empty sequence.  We denote the conditional expectation with respect to $\mathcal{B}_k$ and $\overline{\mathcal{B}}_k$ as 
			\[
			\mathbb{E}_k[ \cdot ] := \mathbb{E} \left[\cdot \mid \mathcal{B}_k \right]
			\]
			and
			\[
			\mathbb{E}_{k,T_k}[ \cdot ] := \mathbb{E} \left[\cdot \mid \overline{\mathcal{B}}_k \right],
			\] 
			respectively.
			Furthermore, we define
			\begin{equation}\label{Q}
				\overline{\mathcal{Q}}_k:=\{T_k \in \overline{\Omega} \mid T_k^{\top} A^{\top} z^k \neq 0\}
				\
				\text{and}
				\
				\mathcal{Q}_k:=\{S_k \in \Omega \mid S_k^{\top}(Ax^k-(b-z^{k+1}))\neq 0\}.
			\end{equation}
			It is evident that $\{\overline{\mathcal{Q}}_k, \overline{\mathcal{Q}}_k^c\}$ and $\{\mathcal{Q}_k, \mathcal{Q}_k^c\}$ form partitions of $\overline{\Omega}$ and $\Omega$, respectively. Given that $T_k \in \overline{Q}$ and $S_k\in Q$, we denote
			\[
			\mathbb{E}_{k, T_{k} \in\overline{\mathcal{Q}}} [ \cdot ]:=\mathop{\mathbb{E}} \left[\cdot \mid \mathcal{B}_{k}, T_{k} \in\overline{\mathcal{Q}} \right],
			\]
			and
			\[
			\mathbb{E}_{k, T_k, S_{k} \in \mathcal{Q}} [ \cdot ]:=\mathop{\mathbb{E}} \left[\cdot \mid \overline{\mathcal{B}}_{k}, S_{k} \in \mathcal{Q} \right].
			\]
			
			Note that the $\sigma$-algebra generated by $\mathcal{B}_k $ is a sub-$\sigma$-algebra of that generated by  $\bar{\mathcal{B}}_k $, by the tower property of conditional expectation \cite[Proposition 12.1.5 (i)]{athreya2006measure}, we have
			\begin{equation}\label{e:tower}
				\mathbb{E}_k[ \cdot ] = \mathbb{E} \left[\cdot \mid \mathcal{B}_k \right] = \mathbb{E} \left[\mathbb{E} \left[\cdot \mid \bar{\mathcal{B}}_k \right] \mid \mathcal{B}_k \right] = \mathbb{E}_k[ \mathbb{E}_{k,T_k}[ \cdot ]].
			\end{equation}
			Besides, for random variables $X$ and $Y$, if $X$ is is measurable with respect to the $\sigma$-algebra generated by $\mathcal{B}_k$, which is denoted by $\sigma \langle \mathcal{B}_k \rangle$, we have the following equations holds \cite[Proposition 12.1.5 (ii)]{athreya2006measure},
			\begin{equation}
				\label{e:measure}
				\mathbb{E} [X|\mathcal{B}_k] = X \quad \text{and} \quad \mathbb{E} [XY|\mathcal{B}_k] = X  \mathbb{E} [Y|\mathcal{B}_k].
			\end{equation} 
			Note that since $z^k$ and $x^k$ are determined only by the elements in the sequence $(T_0,S_0,$ $\ldots,$ $T_{k-1},$ $S_{k-1})$, they are measurable with respect to $\sigma\langle \mathcal{B}_k \rangle$. Similarly, $z^k$, $x^k$, $z^{k+1}$, and $x^{k+1}$ are measurable with respect to $\sigma\langle \overline{\mathcal{B}}_k \rangle$.
			
			We have the following convergence result for the auxiliary sequence $\{z^k\}_{k \geq 0}$ generated by Algorithm \ref{ASE}.
			\begin{lemma}
				\label{z-the basic method-convergence}
				Let $\{z^k\}_{k \geq 0}$ be the iteration sequence generated by Algorithm \ref{ASE}. Suppose that the probability space $(\overline{\Omega}, \overline{\mathcal{F}}, \overline{P})$ satisfies Assumption \ref{Ass}. Then
				$$
				\mathbb{E}[\|z^k-b_{\text{Range}(A)^{\perp}}\|_2^2] \leq \rho_{z}^k \|z^0-b_{\text{Range}(A)^{\perp}}\|_2^2,
				$$
				where $\overline{H}$ and $\rho_{z}$ are defined as \eqref{matrix-H-} and \eqref{rho}, respectively.
			\end{lemma}
			\begin{proof}
				Let $\overline{\mathcal{Q}}_k$ be defined as \eqref{Q} and	if the sampling matrix $T_k \in \overline{\mathcal{Q}}_k$,  we have
				\begin{align}
					\|z^{k+1}-b_{\text{Range}(A)^{\bot}}\|_2^2
					=&\|z^k-\mu_k AT_kT_k^{\top}A^{\top}z^k-b_{\text{Range}(A)^{\bot}}\|_2^2 \nonumber \\
					=&\|z^k-b_{\text{Range}(A)^{\bot}}\|_2^2-2\mu_k \langle z^k-b_{\text{Range}(A)^{\bot}}, AT_kT_k^{\top}A^{\top}z^k \rangle
					\nonumber \\
					&+\mu_k^2 \|AT_kT_k^{\top}A^{\top}z^k\|_2^2 \nonumber \\
					=&\|z^k-b_{\text{Range}(A)^{\bot}}\|_2^2-2(2-\eta)\overline{L}_{\text{adap}}^{(k)} \|T_k^{\top}A^{\top}z^k\|_2^2
					\nonumber \\
					&+(2-\eta)^2 \overline{L}_{\text{adap}}^{(k)} \|T_k^{\top}A^{\top}z^k\|_2^2 \nonumber \\
					=&\|z^k-b_{\text{Range}(A)^{\bot}}\|_2^2-\eta(2-\eta) \overline{L}_{\text{adap}}^{(k)} \|T_k^{\top}A^{\top}z^k\|_2^2 \nonumber \\
					\leq& \|z^k-b_{\text{Range}(A)^{\bot}}\|_2^2-\eta(2-\eta) \frac{\|T_k^{\top}A^{\top}z^k\|_2^2}{\|AT_k\|_2^2}, \nonumber
				\end{align}
				where the third equality follows from the definition of $\mu_k$ and $\overline{L}_{\text{adap}}^{(k)}$ in \eqref{mu} and \eqref{overline_L_adapt}, respectively, and the last inequality follows from the fact that
				$$
				\overline{L}_{\text{adap}}^{(k)}=\frac{\|T_k^{\top}A^{\top}z^k\|_2^2}{\|AT_kT_k^{\top}A^{\top}z^k\|_2^2} \geq \frac{1}{\|AT_k\|_2^2}.
				$$
				Thus
				$$
				\begin{aligned}
					&\mathbb{E}_{k, T_k \in \overline{\Omega}}[\|z^{k+1}-b_{\text{Range}(A)^{\bot}}\|_2^2] \\
					=&
					\mathbb{P}(T_k \in \overline{\mathcal{Q}}_k) \mathbb{E}_{k, T_k \in \overline{\mathcal{Q}}_k}[\|z^{k+1}-b_{\text{Range}(A)^{\bot}}\|_2^2]+\mathbb{P}(T_k \in \overline{\mathcal{Q}}_k^c) \mathbb{E}_{k, T_k \in \overline{\mathcal{Q}}_k^c}[\|z^{k+1}-b_{\text{Range}(A)^{\bot}}\|_2^2] \\
					\leq&
					\mathbb{P}(T_k \in \overline{\mathcal{Q}}_k) \mathbb{E}_{k, T_k \in \overline{\mathcal{Q}}_k}\left[\|z^k-b_{\text{Range}(A)^{\bot}}\|_2^2-\eta(2-\eta) \frac{\|T_k^{\top}A^{\top}z^k\|_2^2}{\|AT_k\|_2^2}\right]
					\\
					&+\mathbb{P}(T_k \in \overline{\mathcal{Q}}_k^c) \mathbb{E}_{k, T_k \in \overline{\mathcal{Q}}_k^c}[\|z^k-b_{\text{Range}(A)^{\bot}}\|_2^2] \\
					= & \mathbb{E}_{k, T_k \in \overline{\Omega}}[\|z^k-b_{\text{Range}(A)^{\bot}}\|_2^2] -\eta(2-\eta) \mathbb{P}(T_k \in \overline{\mathcal{Q}}_k) \mathbb{E}_{k, T_k \in \overline{\mathcal{Q}}_k}\left[\frac{\|T_k^{\top}A^{\top}z^k\|_2^2}{\|AT_k\|_2^2} \right] \\
					= & \|z^k-b_{\text{Range}(A)^{\bot}}\|_2^2 -\eta(2-\eta) \mathbb{P}(T_k \in \overline{\mathcal{Q}}_k) \mathbb{E}_{k, T_k \in \overline{\mathcal{Q}}_k}\left[\frac{\|T_k^{\top}A^{\top}z^k\|_2^2}{\|AT_k\|_2^2} \right],
				\end{aligned}
				$$
				where the third equality follows from the fact that $z^k$ is measurable with respect to $\sigma\langle\mathcal{B}_k\rangle$ and \eqref{e:measure}. Besides, note that $\mathbb{E}_{k, T_k \in \overline{\mathcal{Q}}_k^c}\left[\frac{\|T_k^{\top}A^{\top}z^k\|_2^2}{\|AT_k\|_2^2}\right]=0$ as $T_k^{\top}A^{\top}z^k=0$ for $T_k \in \overline{\mathcal{Q}}_k^c$, we have
				\[
				\begin{aligned}
					&\mathbb{P}(T_k \in \overline{\mathcal{Q}}_k) \mathbb{E}_{k, T_k \in \overline{\mathcal{Q}}_k}\left[\frac{\|T_k^{\top}A^{\top}z^k\|_2^2}{\|AT_k\|_2^2} \right] \\
					=&\mathbb{P}(T_k \in \overline{\mathcal{Q}}_k) \mathbb{E}_{k, T_k \in \overline{\mathcal{Q}}_k}\left[\frac{\|T_k^{\top}A^{\top}z^k\|_2^2}{\|AT_k\|_2^2} \right]
					+ \mathbb{P}(T_k \in \overline{\mathcal{Q}}_k^c) \mathbb{E}_{k, T_k \in \overline{\mathcal{Q}}_k^c}\left[\frac{\|T_k^{\top}A^{\top}z^k\|_2^2}{\|AT_k\|_2^2}\right] \\
					= & \mathbb{E}_{k, T_k \in \overline{\Omega}}\left[\frac{\|T_k^{\top}A^{\top}z^k\|_2^2}{\|AT_k\|_2^2}\right].
				\end{aligned}
				\]
				Therefore, 
				$$
				\begin{aligned}
					\mathbb{E}_{k, T_k \in \overline{\Omega}}[\|z^{k+1}-b_{\text{Range}(A)^{\bot}}\|_2^2]
					= & \|z^k-b_{\text{Range}(A)^{\bot}}\|_2^2 -\eta(2-\eta)\mathbb{E}_{k, T_k \in \overline{\Omega}}\left[\frac{\|T_k^{\top}A^{\top}z^k\|_2^2}{\|AT_k\|_2^2}\right] \\
					=&\|z^k-b_{\text{Range}(A)^{\bot}}\|_2^2 -\eta(2-\eta) \|\overline{H}^{\frac{1}{2}} A^{\top}z^k\|_2^2 \\
					=&\|z^k-b_{\text{Range}(A)^{\bot}}\|_2^2 -\eta(2-\eta) \|\overline{H}^{\frac{1}{2}} A^{\top}(z^k-b_{\text{Range}(A)^{\bot}})\|_2^2 \\
					\leq& \rho_{z} \|z^k-b_{\text{Range}(A)^{\perp}}\|_2^2,
				\end{aligned}
				$$
				where the second equality follows from the fact that $z^k$ is measurable with respect to $\sigma\langle\mathcal{B}_k\rangle$ and \eqref{e:measure}, 
				and the last inequality follows from $\overline{H}$ is positive definite and $z^k-b_{\text{Range}(A)^{\perp}} \in \text{Range}(A)$.
				By taking the full expectation on both sides, we have
				$$
				\mathbb{E}[\|z^k-b_{\text{Range}(A)^{\perp}}\|_2^2] \leq \rho_{z}^k \|z^0-b_{\text{Range}(A)^{\perp}}\|_2^2
				$$
				as desired.
			\end{proof}
			
			Now, we are ready to prove Theorem \ref{x-ASEM-convergence}.
			
			\begin{proof}[Proof of Theorem \ref{x-ASEM-convergence}]
				Let $\mathcal{Q}_k$ be defined as \eqref{Q} and if the sampling matrix $S_k \in \mathcal{Q}_k$, we have
				\begin{equation}
					\label{x^{k+1}_Q_k}
					\begin{aligned}
						\|x^{k+1}-A^{\dagger}b\|_2^2=& \|x^k-\alpha_kA^{\top}S_kS_k^{\top}(Ax^k-(b-z^{k+1}))-A^{\dagger}b\|_2^2
						\\
						=&\|x^k-A^{\dagger}b\|_2^2+ \alpha_k^2 \|A^{\top}S_kS_k^{\top}(Ax^k-(b-z^{k+1}))\|_2^2  \\
						&-2\alpha_k \langle S_k^{\top}A(x^k-A^{\dagger}b), S_k^{\top}(Ax^k-(b-z^{k+1})) \rangle
						\\
						=&\|x^k-A^{\dagger}b\|_2^2+ (2-\zeta) \alpha_k \|S_k^{\top}(Ax^k-(b-z^{k+1}))\|_2^2 \\
						&-2\alpha_k \langle S_k^{\top}A(x^k-A^{\dagger}b), S_k^{\top}(Ax^k-(b-z^{k+1})) \rangle
						\\
						=&\|x^k-A^{\dagger}b\|_2^2+(2-\zeta) \alpha_k\|S_k^{\top}(Ax^k-(b-z^{k+1}))\|_2^2  \\
						&+\alpha_k \left(\|S_k^{\top}(z^{k+1}-b_{\text{Range}(A)^{\bot}})\|_2^2-\|S_k^{\top}A(x^k-A^{\dagger}b)\|_2^2\right.\\
						&\left.  \ \ \ \ \ \ \ \ \ \ \ -\|S_k^{\top}(Ax^k-(b-z^{k+1}))\|_2^2\right)
						\\
						=&\|x^k-A^{\dagger}b\|_2^2-\alpha_k \|S_k^{\top}A(x^k-A^{\dagger}b)\|_2^2+\alpha_k \|S_k^{\top}(z^{k+1}-b_{\text{Range}(A)^{\bot}})\|_2^2  \\
						&-(\zeta-1) \alpha_k\|S_k^{\top}(Ax^k-(b-z^{k+1}))\|_2^2,
					\end{aligned}
				\end{equation}
				where the third equality follows from the definition of $\alpha_k$ in \eqref{alpha_basic}. For any $0 < \epsilon \leq 1$, we can get
				\begin{equation}
					\label{zeta_0,1}
					\begin{aligned}
						\|S_k^{\top}(Ax^k-(b-z^{k+1}))\|_2^2
						=&
						\|S_k^{\top}(Ax^k-AA^{\dagger}b+AA^{\dagger}b-(b-z^{k+1}))\|_2^2 \\
						=&
						\|S_k^{\top}A(x^k-A^{\dagger}b)+S_k^{\top}(z^{k+1}-b_{\text{Range}(A)^{\bot}})\|_2^2 \\
						\leq&
						(1+\epsilon) \|S_k^{\top}A(x^k-A^{\dagger}b)\|_2^2+(\epsilon^{-1}+1) \|S_k^{\top}(z^{k+1}-b_{\text{Range}(A)^{\bot}})\|_2^2,
					\end{aligned}
				\end{equation}
				and
				\begin{equation}
					\label{zeta_1,2}
					\begin{aligned}
						\|S_k^{\top}(Ax^k-(b-z^{k+1}))\|_2^2
						\geq
						(1-\epsilon) \|S_k^{\top}A(x^k-A^{\dagger}b)\|_2^2-(\epsilon^{-1}-1) \|S_k^{\top}(z^{k+1}-b_{\text{Range}(A)^{\bot}})\|_2^2. 
					\end{aligned}
				\end{equation}
				For the case where $\zeta \in (0, 1)$, substituting \eqref{zeta_0,1} into \eqref{x^{k+1}_Q_k}, we have
				\begin{equation}\nonumber
					\begin{aligned}
						\|x^{k+1}-A^{\dagger}b\|_2^2 
						\leq&
						\|x^k-A^{\dagger}b\|_2^2-\left( 1-(1+\epsilon)(1-\zeta)\right) \alpha_{k}  \left\|S_k^{\top}A(x^k-A^{\dagger}b)\right\|_2^2\\
						&+\left( 1+(\epsilon^{-1}+1)(1-\zeta)\right) \alpha_{k} \|S_k^{\top}(z^{k+1}-b_{\text{Range}(A)^{\bot}})\|_2^2\\
						=&
						\|x^k-A^{\dagger}b\|_2^2-c_{\zeta, \epsilon} L_{\text{adap}}^{(k)} \|S_k^{\top}A(x^k-A^{\dagger}b)\|_2^2
						\\
						&+d_{\zeta, \epsilon} L_{\text{adap}}^{(k)} \|S_k^{\top}(z^{k+1}-b_{\text{Range}(A)^{\bot}})\|_2^2,
					\end{aligned}
				\end{equation}
				where $c_{\zeta, \epsilon}$ and $d_{\zeta, \epsilon}$ are given by \eqref{c} and \eqref{d}, respectively.
				For the case where $\zeta \in [1,2)$, substituting \eqref{zeta_1,2} into \eqref{x^{k+1}_Q_k}, we have
				\begin{equation}\nonumber
					\begin{aligned}
						\|x^{k+1}-A^{\dagger}b\|_2^2 
						\leq&
						\|x^k-A^{\dagger}b\|_2^2-\left( 1+(1-\epsilon)(\zeta-1)\right) \alpha_{k}  \|S_k^{\top}A(x^k-A^{\dagger}b)\|_2^2\\
						&+\left( 1+(\epsilon^{-1}-1)(\zeta-1)\right) \alpha_{k} \|S_k^{\top}(z^{k+1}-b_{\text{Range}(A)^{\bot}})\|_2^2\\
						=&
						\|x^k-A^{\dagger}b\|_2^2-c_{\zeta, \epsilon} L_{\text{adap}}^{(k)}  \|S_k^{\top}A(x^k-A^{\dagger}b)\|_2^2
						\\
						&+d_{\zeta, \epsilon} L_{\text{adap}}^{(k)} \|S_k^{\top}(z^{k+1}-b_{\text{Range}(A)^{\bot}})\|_2^2.
					\end{aligned}
				\end{equation}
				%where $c_{\zeta, \epsilon} > 0$ and $d_{\zeta, \epsilon}$ are given by \eqref{c} and \eqref{d}, respectively.
				Thus, for any $\zeta \in (0,2)$, it holds that
				\begin{equation}
					\label{con-exp}
					\begin{aligned}
						&\mathbb{E}_{k, T_{k}, S_k \in \Omega} [\|x^{k+1}-A^{\dagger}b\|_2^2 ] \\
						=&
						\mathbb{P}(S_k \in \mathcal{Q}_k^c) \mathbb{E}_{k, T_{k}, S_k \in \mathcal{Q}_k^c} [\|x^{k+1}-A^{\dagger}b\|_2^2 ]+\mathbb{P}(S_k \in \mathcal{Q}_k) \mathbb{E}_{k, T_{k}, S_k \in \mathcal{Q}_k} [\|x^{k+1}-A^{\dagger}b\|_2^2 ]
						\\
						\leq&
						\mathbb{P}(S_k \in \mathcal{Q}_k^c) \mathbb{E}_{k, T_{k}, S_k \in \mathcal{Q}_k^c} [\|x^k-A^{\dagger}b\|_2^2 ]+\mathbb{P}(S_k \in \mathcal{Q}_k) \mathbb{E}_{k, T_{k}, S_k \in \mathcal{Q}_k} [ \|x^k-A^{\dagger}b\|_2^2 ]  \\
						&-c_{\zeta, \epsilon} \mathbb{P}(S_k \in \mathcal{Q}_k) \mathbb{E}_{k, T_{k}, S_k \in \mathcal{Q}_k}\left[L_{\text{adap}}^{(k)}  \|S_k^{\top}A(x^k-A^{\dagger}b)\|_2^2 \right] \\
						&+d_{\zeta, \epsilon} \mathbb{P}(S_k \in \mathcal{Q}_k) \mathbb{E}_{k, T_{k}, S_k \in \mathcal{Q}_k}\left[L_{\text{adap}}^{(k)} \|S_k^{\top}(z^{k+1}-b_{\text{Range}(A)^{\bot}})\|_2^2 \right]\\
						=&\|x^k-A^{\dagger}b\|_2^2
						-c_{\zeta, \epsilon}
						\underbrace{\mathbb{P}(S_k \in \mathcal{Q}_k) \mathbb{E}_{k, T_{k}, S_k \in \mathcal{Q}_k}[L_{\text{adap}}^{(k)} \|S_k^{\top}A(x^k-A^{\dagger}b)\|_2^2 ]}_{\textcircled{a}}\\
						&+ d_{\zeta, \epsilon}
						\underbrace{\mathbb{P}(S_k \in \mathcal{Q}_k) \mathbb{E}_{k, T_{k}, S_k \in \mathcal{Q}_k}[L_{\text{adap}}^{(k)} \|S_k^{\top}(z^{k+1}-b_{\text{Range}(A)^{\bot}})\|_2^2 ]}_{\textcircled{b}},
					\end{aligned}
				\end{equation}
				where the last equality follows from the fact that $x^{k}$ is measurable with respect to $\sigma\langle \overline{\mathcal{B}}_k \rangle$ and \eqref{e:measure}. Next, we analyze the two expressions $\textcircled{a}$ and $\textcircled{b}$, respectively. Since
				$$
				L_{\text{adap}}^{(k)}=\frac{\left\|S_k^{\top}\left(Ax^k-(b-z^{k+1})\right)\right\|_2^2}{\left\|A^{\top}S_kS_k^{\top}\left(Ax^k-(b-z^{k+1})\right)\right\|_2^2}\geq \frac{1}{\|A^\top S_k\|_2^2},
				$$
				we can establish a lower bound for the first expression as
				\begin{equation}
					\label{S_Q}
					\begin{aligned}
						\textcircled{a}
						\geq&
						\mathbb{P}(S_k \in \mathcal{Q}_k) \mathbb{E}_{k, T_{k}, S_k \in \mathcal{Q}_k}\left[\frac{\left\|S_k^{\top}A(x^k-A^{\dagger}b)\right\|_2^2}{\|A^\top S_k\|_2^2} \right] \\  
						=&
						\mathbb{E}_{k, T_{k}, S_k \in \Omega}\left[\frac{\left\|S_k^{\top}A(x^k-A^{\dagger}b)\right\|_2^2}{\|A^\top S_k\|_2^2} \right]-\mathbb{P}(S_k \in \mathcal{Q}_k^{c}) \mathbb{E}_{k, T_{k}, S_k \in \mathcal{Q}_k^c}\left[\frac{\left\|S_k^{\top}A(x^k-A^{\dagger}b)\right\|_2^2}{\|A^\top S_k\|_2^2} \right]
						\\
						\geq&
						\sigma_{\min}^2(H^{\frac{1}{2}}A)\|x^k-A^{\dagger}b\|_2^2-\mathbb{P}(S_k \in \mathcal{Q}_k^{c}) \mathbb{E}_{k, T_{k}, S_k \in \mathcal{Q}_k^c}\left[\frac{\left\|S_k^{\top}A(x^k-A^{\dagger}b)\right\|_2^2}{\|A^\top S_k\|_2^2} \right], 
					\end{aligned}
				\end{equation}
				where the last inequality follows from $H=\mathop{\mathbb{E}}\limits_{S \in \Omega}\left[SS^{\top} / \left\|A^{\top} S\right\|_2^2\right]$ is positive definite and $x^k-A^{\dag}b \in \text{Range}(A^\top)$. Furthermore, note that if $S_k \in\mathcal{Q}_k^c$, i.e. $S_k^{\top}\left(Ax^k-(b-z^{k+1})\right)=0$, it holds that
				\begin{equation}\nonumber
					\begin{aligned}
						S_k^{\top}A(x^k-A^{\dagger}b)
						=&
						S_k^{\top}\left(Ax^k-(b-z^{k+1})+(b-z^{k+1})-AA^{\dagger}b\right)\\
						=&
						S_k^{\top}\left(Ax^k-(b-z^{k+1})\right)+S_k^{\top}(b_{\text{Range}(A)^{\bot}}-z^{k+1})\\
						=&
						S_k^{\top}(b_{\text{Range}(A)^{\bot}}-z^{k+1}).
					\end{aligned}
				\end{equation}
				Substituting it into \eqref{S_Q}, we can get
				\begin{equation}\label{S_Q_1}
					\begin{aligned}
						\textcircled{a}
						\geq&
						\sigma_{\min}^2(H^{\frac{1}{2}}A)\|x^k-A^{\dagger}b\|_2^2-\mathbb{P}(S_k \in \mathcal{Q}_k^{c}) \mathbb{E}_{k, T_{k}, S_k \in \mathcal{Q}_k^c}\left[\frac{\left\|S_k^{\top}(z^{k+1}-b_{\text{Range}(A)^{\bot}})\right\|_2^2}{\|A^\top S_k\|_2^2} \right].
					\end{aligned}
				\end{equation}
				In addition, since
				$$
				\begin{aligned}
					L_{\text{adap}}^{(k)}
					&=
					\frac{\left\|S_k^{\top}\left(Ax^k-(b-z^{k+1})\right)\right\|_2^2}{\left\|A^{\top}S_kS_k^{\top}\left(Ax^k-(b-z^{k+1})\right)\right\|_2^2} =
					\frac{\left\|S_k^{\top}A\left(x^k-A^{\dagger}(b-z^{k+1})\right)\right\|_2^2}{\left\|A^{\top}S_kS_k^{\top}A\left(x^k-A^{\dagger}(b-z^{k+1})\right)\right\|_2^2} \\
					&\leq
					\frac{1}{\lambda_{\min}\left(\frac{A^\top S_kS_k^\top A}{\|A^\top S_k\|_2^2}\right)\|A^\top S_k\|_2^2} \leq \frac{1}{\Lambda_{\min}\|A^\top S_k\|_2^2},
				\end{aligned}
				$$
				where the second equality follows from the fact that $b-z^{k+1}=AA^{\dagger}(b-z^{k+1})$ as $b-z^{k+1} \in \text{Range}(A)$. Hence, the second expression can be bounded as
				\begin{equation}\label{S_Q_2}
					\begin{aligned}
						\textcircled{b}
						\leq
						\frac{1}{\Lambda_{\min}} \mathbb{P}(S_k \in \mathcal{Q}_k) \mathbb{E}_{k, T_{k}, S_k \in \mathcal{Q}_k}\left[ \frac{\|S_k^{\top}(z^{k+1}-b_{\text{Range}(A)^{\bot}})\|_2^2}{\|A^\top S_k\|_2^2} \right].
					\end{aligned}
				\end{equation}
				By the assumptions of this theorem, specifically, if $\zeta \in (0, \frac{1}{2}]$ and $\epsilon \in (0, \frac{\zeta}{1-\zeta})$, or $\zeta \in (\frac{1}{2}, 2)$ and $\epsilon \in (0, 1]$,
				it follows  that $c_{\zeta, \epsilon}> 0$. Furthermore,  for any $\zeta \in (0, 2)$ and $\epsilon \in (0, 1]$, we have $d_{\zeta, \epsilon}>0$. Hence,
				by substituting \eqref{S_Q_1} and \eqref{S_Q_2} into \eqref{con-exp}, and noting that $\rho_{x}$ is defined as \eqref{v}, we have
				\begin{equation}\label{x_Omega}
					\begin{aligned}
						\mathbb{E}_{k, T_{k}, S_k \in \Omega}[\|x^{k+1}-A^{\dagger}b\|_2^2]
						\leq&
						\rho_{x} \|x^k-A^{\dagger}b\|_2^2 \\
						&+c_{\zeta, \epsilon} \mathbb{P}(S_k \in \mathcal{Q}_k^{c}) \mathbb{E}_{k, T_{k}, S_k \in \mathcal{Q}_k^c}\left[\frac{\|S_k^{\top}(z^{k+1}-b_{\text{Range}(A)^{\bot}})\|_2^2}{\|A^\top S_k\|_2^2} \right]\\
						&+\frac{d_{\zeta, \epsilon}}{\Lambda_{\min}} \mathbb{P}(S_k \in \mathcal{Q}_k) \mathbb{E}_{k, T_{k}, S_k \in \mathcal{Q}_k}\left[\frac{\|S_k^{\top}(z^{k+1}-b_{\text{Range}(A)^{\bot}})\|_2^2}{\|A^\top S_k\|_2^2} \right]
						\\
						\leq&
						\rho_{x} \|x^k-A^{\dagger}b\|_2^2+\frac{d_{\zeta, \epsilon}}{\Lambda_{\min}} \mathbb{E}_{k, T_{k}, S_k \in \Omega}\left[\frac{\|S_k^{\top}(z^{k+1}-b_{\text{Range}(A)^{\bot}})\|_2^2}{\|A^\top S_k\|_2^2} \right]
						\\
						\leq& \rho_{x} \|x^k-A^{\dagger}b\|_2^2+ \frac{d_{\zeta, \epsilon} \lambda_{\max}(H)}{\Lambda_{\min}} \|z^{k+1}-b_{\text{Range}(A)^{\bot}}\|_2^2,
					\end{aligned}
				\end{equation}
				where the second inequality follows from the facts that $c_{\zeta, \epsilon} \leq d_{\zeta, \epsilon}$ and $\Lambda_{\min} \leq 1$, and the last inequality holds since $z^{k+1}$ is measurable with respect to $\sigma\langle \overline{\mathcal{B}}_k \rangle$. In total,
				\begin{align*}
					\mathbb{E}_k\left[\|x^{k+1}-A^{\dagger}b\|_2^2\right]=&\mathbb{E}_{k, T_{k} \in \overline{\Omega}} \left[\mathbb{E}_{k, T_{k}, S_k \in \Omega}\left[\|x^{k+1}-A^{\dagger}b\|_2^2\right]\right]
					\\
					\leq & \rho_{x}  \mathbb{E}_{k, T_{k} \in \overline{\Omega}} \|x^k-A^{\dagger}b\|_2^2+ \frac{d_{\zeta, \epsilon} \lambda_{\max}(H)}{\Lambda_{\min}} \mathbb{E}_{k, T_{k} \in \overline{\Omega}}\left[\|z^{k+1}-b_{\text{Range}(A)^{\bot}}\|_2^2\right] \\
					=& \rho_{x} \|x^k-A^{\dagger}b\|_2^2+ \frac{d_{\zeta, \epsilon} \lambda_{\max}(H)}{\Lambda_{\min}} \mathbb{E}_{k, T_{k} \in \overline{\Omega}}\left[\|z^{k+1}-b_{\text{Range}(A)^{\bot}}\|_2^2\right],
				\end{align*}
				where the first equality follows from \eqref{e:tower}, and the last equality holds because $x^k$ is measurable with respect to $\sigma\langle \mathcal{B}_k \rangle$.
				Together with Lemma \ref{z-the basic method-convergence}, we have
				\begin{equation}
					\label{x_ASE_con}
					\begin{aligned}
						\mathbb{E}\left[\|x^{k+1}-A^{\dagger}b\|_2^2\right]
						\leq&
						\rho_{x} \mathbb{E}\left[\|x^k-A^{\dagger}b\|_2^2\right]+\frac{d_{\zeta, \epsilon} \lambda_{\max}(H)}{\Lambda_{\min}} \mathbb{E}\left[\|z^{k+1}-b_{\text{Range}(A)^{\bot}}\|_2^2\right]
						\\
						\leq&
						\rho_{x} \mathbb{E}\left[\|x^k-A^{\dagger}b\|_2^2\right]+\frac{d_{\zeta, \epsilon} \lambda_{\max}(H) \rho_{z}^{k+1}}{\Lambda_{\min}} \|z^0-b_{\text{Range}(A)^{\bot}}\|_2^2
						\\
						\leq& \rho_{x}^2 \mathbb{E}\left[\|x^{k-1}-A^{\dagger}b\|_2^2\right]  \\
						&+ \frac{d_{\zeta, \epsilon} \lambda_{\max}(H)}{\Lambda_{\min}} \|z^0-b_{\text{Range}(A)^{\bot}}\|_2^2 \left(\rho_{z}^k \rho_{x} +\rho_{z}^{k+1}\right)
						\\
						\leq& \rho_{x}^{k+1} \|x^0-A^{\dagger}b\|_2^2+\frac{d_{\zeta, \epsilon} \lambda_{\max}(H)}{\Lambda_{\min}} \|z^0-b_{\text{Range}(A)^{\bot}}\|_2^2 \sum\limits_{i=0}^k \rho_{z}^{k+1-i} \rho_{x}^i. 
					\end{aligned}
				\end{equation}
				For the case where $\rho_{x} \leq \rho_{z}$, we have
				$$
				\sum\limits_{i=0}^k \rho_{z}^{k+1-i} \rho_{x}^i=\rho_{z}^{k+1} \sum\limits_{i=0}^k \left(\frac{\rho_{x}}{\rho_{z}}\right)^i \leq \frac{\rho_{z}^{k+1}}{\max\left\{1-\frac{\rho_{x}}{\rho_{z}}, \frac{1}{k+1}\right\}}=\frac{\rho^{k+1}}{\max\left\{\left|1-\frac{\rho_{x}}{\rho_{z}}\right|, \frac{1}{k+1}\right\}},
				$$
				where $\rho:=\max\{\rho_{x}, \rho_{z}\}$. In addtion, for the case where $\rho_{z} < \rho_{x}$, we have
				$$
				\sum\limits_{i=0}^k \rho_{z}^{k+1-i} \rho_{x}^i=\rho_{x}^{k+1} \sum\limits_{i=0}^k \left(\frac{\rho_{z}}{\rho_{x}}\right)^{k+1-i} \leq \frac{\rho_{x}^{k+1}}{\max\left\{\frac{\rho_{x}}{\rho_{z}}-1, \frac{1}{k+1}\right\}} \leq \frac{\rho^{k+1}}{\max\left\{\left|1-\frac{\rho_{x}}{\rho_{z}}\right|, \frac{1}{k+1}\right\}}.
				$$
				Hence the inequality $\sum\limits_{i=0}^k \rho_{z}^{k+1-i} \rho_{x}^i \leq \frac{\rho^{k+1}}{\max\left\{\left|1-\frac{\rho_{x}}{\rho_{z}}\right|, \frac{1}{k+1}\right\}}$ holds for any $\rho_{z}, \rho_{x} \in (0, 1)$, which together with \eqref{x_ASE_con} implies that
				\begin{align*}
					\mathbb{E}\left[\|x^{k+1}-A^{\dagger}b\|_2^2\right]
					\leq&
					\rho_{x}^{k+1} \|x^0-A^{\dagger}b\|_2^2+\frac{d_{\zeta, \epsilon} \lambda_{\max}(H)}{\Lambda_{\min}} \|z^0-b_{\text{Range}(A)^{\bot}}\|_2^2 \sum\limits_{i=0}^k \rho_{z}^{k+1-i} \rho_{x}^i \\
					\leq&
					\rho^{k+1} \left(\|x^0-A^{\dagger}b\|_2^2+\frac{d_{\zeta, \epsilon} \lambda_{\max}(H)}{\Lambda_{\min}\max\left\{\left|1-\frac{\rho_{x}}{\rho_{z}}\right|, \frac{1}{k+1}\right\}} \|z^0-b_{\text{Range}(A)^{\bot}}\|_2^2\right).
				\end{align*}
				This completes the proof of this theorem.
			\end{proof}
			
			\subsection{Proof of Theorem \ref{the3}}
			
			For convenience, we denote
			$$
			\overline{z}^{k+1}=z^k-\overline{L}_{\text{adap}}^{(k)} AT_kT_k^{\top}A^{\top}z^k, \quad
			\hat{x}^{k+1}=x^k-L_{\text{adap}}^{(k)} A^{\top} S_k S_k^{\top}(Ax^k-(b-z^{k+1})),
			$$
			where $\overline{L}_{\text{adap}}^{(k)}$ and $L_{\text{adap}}^{(k)}$ is given by \eqref{overline_L_adapt} and \eqref{L_adapt}, respectively. We set
			$$
			\overline{\mathcal{Q}}_{k, \neq}:=\{T_k \in \overline{\Omega} \mid \Vert p^k \Vert_2^2 \Vert z^k-z^{k-1} \Vert_2^2
			-\langle p^k,  z^k-z^{k-1}\rangle^2 \neq 0\},
			$$
			and
			$$
			\mathcal{Q}_{k, \neq}:=
			\{S_k \in \Omega \mid \Vert q^k \Vert_2^2 \Vert x^k-x^{k-1} \Vert_2^2
			-\langle q^k,  x^k-x^{k-1}\rangle^2 \neq 0\},
			$$
			where $p^k:=AT_kT_k^{\top}A^{\top}z^k$ and $q^k:=A^{\top}S_kS_k^{\top}(Ax^k-(b-z^{k+1}))$.

			Firstly,  we can
			derive the following convergence result for the iteration sequence $\{z^k\}_{k \geq 0}$ generated by Algorithm \ref{ASEHBM}.
			\begin{lemma}
				\label{z-ASHBM-convergence}
				Let $\{z^k\}_{k \geq 0}$ be the iteration sequence generated by Algorithm \ref{ASEHBM}.  Suppose that the probability space $(\overline{\Omega}, \overline{\mathcal{F}}, \overline{P})$ satisfies Assumption \ref{Ass}. Then
				$$
				\mathbb{E}[\|z^k-b_{\text{Range}(A)^{\perp}}\|_2^2] \leq \left(1-\sigma_{\min}^2(\overline{H}^{\frac{1}{2}}A^{\top})\right)^k \|z^0-b_{\text{Range}(A)^{\perp}}\|_2^2,
				$$
				where $\overline{H}$ is defined as \eqref{matrix-H-}.
			\end{lemma}
			
			\begin{proof}[Proof of Lemma \ref{z-ASHBM-convergence}]
				Since $\{\overline{\mathcal{Q}}_{k, \neq}, \overline{\mathcal{Q}}_k \backslash \overline{\mathcal{Q}}_{k, \neq}, \overline{\mathcal{Q}}_k^c\}$ forms a partition of $\overline{\Omega}$, hence the sampling matrix $T_k \in \overline{\Omega}$ can be classified into the following three cases: $T_k \in \overline{\mathcal{Q}}_{k, \neq}$, $T_k \in \overline{\mathcal{Q}}_k \backslash \overline{\mathcal{Q}}_{k, \neq}$, and $T_k \in \overline{\mathcal{Q}}_k^c$.
				
				{\bf Case 1} $(T_k \in \overline{\mathcal{Q}}_{k, \neq})$: If the sampling matrix $T_k \in \overline{\mathcal{Q}}_{k, \neq}$, according to the definition of $z^{k+1}$, we have
				$$
				\|z^{k+1}-b_{\text{Range}(A)^{\bot}}\|_2^2 \leq \|\overline{z}^{k+1}-b_{\text{Range}(A)^{\bot}}\|_2^2.
				$$
				
				{\bf Case 2} $(T_k \in \overline{\mathcal{Q}}_k \backslash \overline{\mathcal{Q}}_{k, \neq})$: If the sampling matrix $T_k \in \overline{\mathcal{Q}}_k \backslash \overline{\mathcal{Q}}_{k, \neq}$, we have $z^{k+1}=\overline{z}^{k+1}$.
				
				{\bf Case 3} $(T_k \in \overline{\mathcal{Q}}_k^c)$: If the sampling matrix $T_k \in \overline{\mathcal{Q}}_k^c$, we have $z^{k+1}=z^k$.
				
				Therefore, it holds that
				%	\begin{equation}
					$$
					\begin{aligned}
						&\mathbb{E}_{k, T_k \in \overline{\Omega}}[\|z^{k+1}-b_{\text{Range}(A)^{\bot}}\|_2^2] \\
						\leq&
						\mathbb{P}(T_k \in \overline{\mathcal{Q}}_{k, \neq}) \mathbb{E}_{k, T_k \in \overline{\mathcal{Q}}_{k, \neq}}[\|\overline{z}^{k+1}-b_{\text{Range}(A)^{\bot}}\|_2^2]
						\\
						&+\mathbb{P}(T_k \in \overline{\mathcal{Q}}_k \backslash \overline{\mathcal{Q}}_{k, \neq}) \mathbb{E}_{k, T_k \in \overline{\mathcal{Q}}_k \backslash \overline{\mathcal{Q}}_{k, \neq}}[\|\overline{z}^{k+1}-b_{\text{Range}(A)^{\bot}}\|_2^2] \\
						&+\mathbb{P}(T_k \in \overline{\mathcal{Q}}_k^c) \mathbb{E}_{k, T_k \in \overline{\mathcal{Q}}_k^c}[\|z^k-b_{\text{Range}(A)^{\bot}}\|_2^2] \\
						\leq&
						\mathbb{P}(T_k \in \overline{\mathcal{Q}}_k) \mathbb{E}_{k, T_k \in \overline{\mathcal{Q}}_k}[\|\overline{z}^{k+1}-b_{\text{Range}(A)^{\bot}}\|_2^2]+\mathbb{P}(T_k \in \overline{\mathcal{Q}}_k^c)\mathbb{E}_{k, T_k \in \overline{\mathcal{Q}}_k^c}[\|z^k-b_{\text{Range}(A)^{\bot}}\|_2^2] \\
						\leq&
						\left(1- \sigma_{\min}^2(\overline{H}^{\frac{1}{2}}A^{\top})\right) \|z^k-b_{\text{Range}(A)^{\perp}}\|_2^2,
					\end{aligned}
					$$
					%	\end{equation}
				where the last inequality from Lemma \ref{z-the basic method-convergence}.
				By taking the full expectation on both sides, we have
				$$
				\mathbb{E}[\|z^k-b_{\text{Range}(A)^{\perp}}\|_2^2] \leq \left(1- \sigma_{\min}^2(\overline{H}^{\frac{1}{2}}A^{\top})\right)^k \|z^0-b_{\text{Range}(A)^{\perp}}\|_2^2
				$$
				as desired.
			\end{proof}
			
			Now we are ready to prove Theorem \ref{the3}
			
			\begin{proof}[Proof of Theorem \ref{the3}]
				Since $\{\mathcal{Q}_{k, \neq}, \mathcal{Q}_k \backslash \mathcal{Q}_{k, \neq}, \mathcal{Q}_k^c\}$ forms a partition of $\Omega$, the sampling matrix $S_k \in \Omega$ can be classified into the following three cases: $S_k \in \mathcal{Q}_{k, \neq}$, $S_k \in \mathcal{Q}_k \backslash \mathcal{Q}_{k, \neq}$, and $S_k \in \mathcal{Q}_k^c$.
				
				{ \bf Case 1} $(S_k \in \mathcal{Q}_{k, \neq})$: Let $\tilde{x}^{k+1}$ denote the orthogonal projection of $A^{\dagger}b$ onto the affine set $\Pi_k=x^k+\text{Span}\{q_k, x^k-x^{k-1}\}$ and if the sampling matrix $S_k \in \mathcal{Q}_{k, \neq}$, we have $\langle \tilde{x}^{k+1}-A^{\dagger}b, \tilde{x}^{k+1}-x^{k+1} \rangle=0$. In addition, since $\hat{x}^{k+1}:=x^k- L_{\text{adap}}^{(k)} A^{\top}S_k S_k^{\top}(Ax^k-(b-z^{k+1})) \in \Pi_k$, we can get
				\begin{equation}
					\label{x_Q_k_neq}
					\begin{aligned}
						\Vert x^{k+1}-A^{\dagger}b  \Vert_2^2
						=&
						\Vert (x^{k+1}-\tilde{x}^{k+1})+(\tilde{x}^{k+1}-A^{\dagger}b) \Vert_2^2 \\
						=&
						\Vert \tilde{x}^{k+1}-A^{\dagger}b \Vert_2^2+\Vert x^{k+1}-\tilde{x}^{k+1} \Vert_2^2 \\
						\leq&
						\Vert \hat{x}^{k+1}-A^{\dagger}b \Vert_2^2+\Vert x^{k+1}-\tilde{x}^{k+1} \Vert_2^2.
					\end{aligned}
				\end{equation}
				According to the definitions of $x^{k+1}$ and $\tilde{x}^{k+1}$, it can be observed that the vector $x^{k+1}-\tilde{x}^{k+1}$ represents the orthogonal projection of the vector $A^{\dagger}(b-z^{k+1})-A^{\dagger}b=A^{\dagger}(b_{Range(A)^{\bot}}-z^{k+1})$ onto the affine set $\Pi_k$. Hence,
				$$
				\Vert x^{k+1}-\tilde{x}^{k+1} \Vert_2^2 \leq \Vert A^{\dagger}(z^{k+1}-b_{Range(A)^{\bot}}) \Vert_2^2 \leq \frac{\Vert z^{k+1}-b_{Range(A)^{\bot}} \Vert_2^2}{\sigma_{\min}^2(A)}.
				$$
				Substituting it into \eqref{x_Q_k_neq}, we can get
				\begin{align*}
					\Vert x^{k+1}-A^{\dagger}b  \Vert_2^2 \leq \Vert \hat{x}^{k+1}-A^{\dagger}b \Vert_2^2+\frac{\Vert z^{k+1}-b_{Range(A)^{\bot}} \Vert_2^2}{\sigma_{\min}^2(A)}.
				\end{align*}
				
				{\bf Case 2} $(S_k \in \mathcal{Q}_k \backslash \mathcal{Q}_{k, \neq})$: If the sampling matrix $S_k \in \mathcal{Q}_k \backslash \mathcal{Q}_{k, \neq}$, it holds that $x^{k+1}=\hat{x}^{k+1}$.
				
				{\bf Case 3} $(S_k \in \mathcal{Q}_k^c)$: If the sampling matrix $S_k \in \mathcal{Q}_k^c$, it holds that $x^{k+1}=x^k$.
				
				Recall $\hat{\rho}_{x}$ and $\gamma$ are defined as \eqref{rho_z} and \eqref{gamma}, respectively, we can obtain
				$$
				\begin{aligned}
					&\mathbb{E}_{k, T_{k}, S_k \in \Omega}\left[\|x^{k+1}-A^{\dagger}b\|_2^2\right] \\
					=&
					\mathbb{P}(S_k \in \mathcal{Q}_{k, \neq}) \mathbb{E}_{k, T_{k}, S_k \in \mathcal{Q}_{k, \neq}}\left[\|x^{k+1}-A^{\dagger}b\|_2^2\right] \\
					&+\mathbb{P}(S_k \in \mathcal{Q}_k \backslash \mathcal{Q}_{k, \neq}) \mathbb{E}_{k, T_{k}, S_k \in \mathcal{Q}_k \backslash \mathcal{Q}_{k, \neq}}\left[\|x^{k+1}-A^{\dagger}b\|_2^2\right] \\
					&+\mathbb{P}(S_k \in \mathcal{Q}_k^c) \mathbb{E}_{k, T_{k}, S_k \in \mathcal{Q}_k^c}\left[\|x^{k+1}-A^{\dagger}b\|_2^2\right] \\
					\leq&
					\mathbb{P}(S_k \in \mathcal{Q}_{k}) \mathbb{E}_{k, T_{k}, S_k \in \mathcal{Q}_k}\left[\|\hat{x}^{k+1}-A^{\dagger}b\|_2^2\right]+\mathbb{P}(S_k \in \mathcal{Q}_k^c) \mathbb{E}_{k, T_{k}, S_k \in \mathcal{Q}_k^c}\left[\|x^k-A^{\dagger}b\|_2^2\right] \\
					&+\mathbb{P}(S_k \in \mathcal{Q}_{k, \neq}) \frac{\Vert z^{k+1}-b_{Range(A)^{\bot}} \Vert_2^2}{\sigma_{\min}^2(A)} \\
					\leq&
					\hat{\rho}_x \|x^k-A^{\dagger}b\|_2^2+ \frac{\lambda_{\max}(H)}{\Lambda_{\min}} \|z^{k+1}-b_{\text{Range}(A)^{\bot}}\|_2^2+\frac{\Vert z^{k+1}-b_{Range(A)^{\bot}} \Vert_2^2}{\sigma_{\min}^2(A)} \\
					=&
					\hat{\rho}_x \|x^k-A^{\dagger}b\|_2^2+ \gamma \|z^{k+1}-b_{\text{Range}(A)^{\bot}}\|_2^2,
				\end{aligned}
				$$
				where the second inequality follows from the definition of $\hat{x}^{k+1}$ and \eqref{x_Omega}. Then, according to Lemma \ref{z-ASHBM-convergence} and the proof of Theorem \ref{x-ASEM-convergence}, define $\hat{\rho}:=\max\{\hat{\rho}_{z}, \hat{\rho}_{x}\}$, we have
				\begin{align*}
					\mathbb{E}\left[\|x^{k+1}-A^{\dagger}b\|_2^2\right]
					\leq&
					\hat{\rho}_{x}^{k+1} \|x^0-A^{\dagger}b\|_2^2+ \gamma \|z^0-b_{\text{Range}(A)^{\bot}}\|_2^2 \sum\limits_{i=0}^k \hat{\rho}_{z}^{k+1-i} \hat{\rho}_{x}^i \\
					\leq&
					\hat{\rho}^{k+1} \left(\|x^0-A^{\dagger}b\|_2^2+\frac{\gamma}{\max\left\{\left|1-\frac{\hat{\rho}_{x}}{\hat{\rho}_{z}}\right|, \frac{1}{k+1}\right\}} \|z^0-b_{\text{Range}(A)^{\bot}}\|_2^2\right).
				\end{align*}
				This completes the proof of this theorem.
			\end{proof}
				
			\end{CJK}
		\end{document}